%% file: main.tex
\documentclass[11pt,a4paper,final]{article}
\include{commands}

\begin{document}

\title{A new approach to handle curved meshes in the hybrid high-order method}
\author{Liam Yemm}
\affil{School of Mathematics, Monash University, Melbourne, Australia, \email{liam.yemm@monash.edu}}
\maketitle

\begin{abstract}
  We present here a novel approach to handling curved meshes in polytopal methods within the framework of hybrid high-order methods. The hybrid high-order method is a modern numerical scheme for the approximation of elliptic PDEs. An extension to curved meshes allows for the strong enforcement of boundary conditions on curved domains, and for the capture of curved geometries that appear internally in the domain e.g. discontinuities in a diffusion coefficient. The method makes use of non-polynomial functions on the curved faces and does not require any mappings between reference elements/faces. Such an approach does not require the faces to be polynomial, and has a strict upper bound on the number of degrees of freedom on a curved face for a given polynomial degree. Moreover, this approach of enriching the space of unknowns on the curved faces with non-polynomial functions should extend naturally to other polytopal methods. We show the method to be stable and consistent on curved meshes and derive optimal error estimates in $L^2$ and energy norms. We present numerical examples of the method on a domain with curved boundary, and for a diffusion problem such that the diffusion tensor is discontinuous along a curved arc.  
  \medskip\\
  \textbf{Key words:} hybrid high-order methods, curved meshes, error estimates, numerical tests, polytopal methods. 
  \medskip\\
  \textbf{MSC2010:} 65N12, 65N15, 65N30.
\end{abstract}

\section{Introduction}

In recent years, there has been a trend in the computational literature towards arbitrary order polytopal methods for the approximation of partial differential equations. Such methods have a greater flexibility in the mesh requirements and can capture more intricate geometric and physical details in the domain. Being of arbitrary order, they also benefit from better convergence rates with respect to the global degrees of freedom. A short list of such methods includes discontinuous Galerkin and hybridizable discontinuous Galerkin methods \cite{di-pietro.ern:2011:mathematical,cangiani.dong.ea:2017:discontinuous,cockburn.dong.ea:2009:hybridizable}, virtual element methods \cite{da-veiga.brezzi.ea:2013:basic,ahmad.alsaedi.ea:2013:equivalent,brezzi.falk.ea:2014:basic,cangiani.manzini.ea:2017:conforming}, weak Galerkin methods \cite{mu.wang.ea:2015:weak}, and polytopal finite elements \cite{sukumar.tabarraei:2004:conforming}. However, it is well known that any approximation method on a polytopal mesh of a smooth domain (i.e. with a first order representation of the boundary) will yield at best an order two convergence rate \cite{thomee:1971:polygonal,strang.berger:1973:change}. Thus, any high-order method on curved domains requires a high-order (or exact) representation of the boundary for optimal convergence. 

Developed in \cite{di-pietro.ern.ea:2014:arbitrary, di-pietro.ern:2015:hybrid}, hybrid high-order (HHO) schemes are modern polytopal methods for the approximation of elliptic PDEs. A key aspect of HHO is its applicability to generic meshes with arbitrarily shaped polytopal elements. This article focuses on the extension of HHO methods to allow for curved meshes, with unknowns that capture the geometry exactly, yet still achieve optimal convergence. While the approach is presented within an HHO framework for a diffusion problem, the key ideas are more general and can be extended to related polytopal methods and to other models such as linear elasticity, or the Stokes and Navier--Stokes equations.

There has been much work on the development of discontinuous Galerkin (DG) methods on curved meshes \cite {cangiani.dong.ea:2022:hp,cangiani.georgoulis.ea:2018:adaptive,hindenlang.bolemann.ea:2015:mesh}. We also make note of the article \cite{sevilla.fernandez.ea:2011:comparison} which analyses several approaches to high-order finite element methods on curved meshes. However, for the aforementioned methods the problem is much simpler than for hybrid high-order methods due to the lack of unknowns on the mesh faces. The addition of unknowns on the mesh faces is one of the key benefits hybrid methods have over DG and other non-hybrid methods due to the strong enforcement of boundary conditions and the reduction of the global degrees of freedom via static condensation \cite[Appendix B.3.2]{di-pietro.droniou:2020:hybrid}.

The article \cite{da-veiga.russo.ea:2019:virtual} proposes a virtual element method (VEM) in two dimensions for meshes possessing curved edges. For each curved edge the authors consider the space of polynomials on a linear reference segment in $\bbR$ and map this space onto the curved edge via a sufficiently smooth parameterisation. A similar approach is taken in the articles \cite{artioli.da-veiga.ea:2020:adaptive, dassi.fumagalli.ea:2021:mixed}. A typical approach for hybridizable discontinuous Galerkin (HDG) methods on curved domains is to map the boundary data onto a polytopal sub-domain  \cite[cf.][]{cockburn.solano:2012:solving,cockburn.qiu.ea:2014:apriori}. We make note of the articles \cite{solano.vargas:2022:unfitted, solano.terrana.ea:2022:hdg} which also use this approach to curved boundaries.

While there has been some work on the development of hybrid high-order methods on curved meshes \cite[][]{botti.di-pietro:2018:assessment, burman.ern:2018:unfitted, burman.cicuttin.ea:2021:unfitted}, the approach we take in this paper is quite different. Indeed, the article \cite{botti.di-pietro:2018:assessment} approaches the issue of defining unknowns on curved faces by considering a polynomial mapping from a planar reference face onto the curved face. While this naturally requires the mesh faces to be polynomial, it also reduces the approximation order \cite{botti:2012:influence}. Indeed, if the mapping onto the face has effective mapping order $m$ (see \cite[Equation (5) \& Remark 1]{botti.di-pietro:2018:assessment}), then defining face unknowns of degree $l$ in the reference frame will yield approximation properties of at best order $\lfloor \frac{l}{m} \rfloor$ \cite[Equation (8)]{botti.di-pietro:2018:assessment}.  To recover the optimal approximation order observed for straight meshes, the degree of the face polynomials in the reference frame is increased by a factor of $m$, yielding a very large global stencil for high-order mappings. Moreover, approximation properties in the curved faces are unknown, and the authors assume them to be true \cite[cf.][Equation (9)]{botti.di-pietro:2018:assessment} in order to obtain optimal error estimates. We also make note of the conference proceedings \cite{de-souza.loula:primal:2020} which follows the same approach using reference frame polynomials to define unknowns on curved faces for a HDG method.

An alternative approach, first considered in \cite{burman.ern:2018:unfitted, burman.cicuttin.ea:2021:unfitted}, is to increase the polynomial degree of the element unknowns and weakly enforce the boundary or interface conditions without defining any unknowns on curved faces. This procedure has also been implemented for a fourth order bi-harmonic problem in a curved domain \cite{dong.ern:2021:hybrid}. Such an approach ensures stability of the system and that optimal convergence rates are achieved. However, this method does not capture the geometry exactly and requires a finely tuned Nitsche parameter to achieve stability and consistency \cite[cf.][]{prenter.lehrenfeld.ea:2018:note}. Moreover, without unknowns defined on curved faces, it is not clear how to design an enriched method such as that proposed in \cite{yemm:2022:design}, whereas the method devised in this paper works seamlessly with enrichment.

In this paper we take inspiration from the article \cite{yemm:2022:design} and consider unknowns on the faces to include the Neumann traces of higher order polynomials. We note that this approach does not consider reference elements or faces but rather directly defines \emph{non}-polynomial spaces on curved faces. Such an approach is therefore more closely analogous to an enriched or extended method than it is to any of the previously mentioned methods of defining unknowns on curved faces. Using this approach, we are not restricted to consider polynomial faces, but can rather take any $C^1$ manifold. Moreover, the number of degrees of freedom on curved faces is strictly bounded above and does not grow arbitrarily large for high-order mappings. We are able to prove consistency of the scheme, and, by including the space of constant functions on the faces the method is shown to be stable. In Section \ref{sec:error} we prove optimal error estimates in energy and $L^2$-norm, and in Section \ref{sec:integration} we present a method for the design of quadrature rules on curved elements. The paper is concluded with some numerical tests in two dimensions in Section \ref{sec:implementation}.

\subsection{Model and Assumptions on the Mesh}\label{sec:model.and.mesh}

We take a domain \(\Omega\subset \R^d\), \(d\ge 2\), and consider the Dirichlet--diffusion problem: find \(u\in \HONEzr( \Omega) \) such that
\begin{equation}\label{eq:weak.form}
	\a(u, v)  = \ell(v), \qquad\forall v\in \HONEzr(\Omega), 
\end{equation}
where \(\a(u, v) \defeq \brac[\Omega]{\matK\nabla u, \nabla v}\) and \(\ell(v) \defeq \brac[\Omega]{f, v}\) for some source term \(f\in \LTWO(\Omega)\) and diffusion tensor \(\matK\) assumed to be a symmetric, piecewise constant matrix-valued function satisfying, for all \(\bmx\in\R^d\),
\begin{equation}\label{eq:uniform.elliptic}
	\ulK\bmx\cdot\bmx \le (\matK\bmx)\cdot\bmx \le \olK\bmx\cdot\bmx 
\end{equation}
for two fixed real numbers \(0 < \ulK \le \olK\). Here and in the following, \(\brac[X]{\cdot, \cdot}\) is the \(\LTWO\)-inner product of scalar- or vector-valued functions on a set \(X\) for its natural measure. We shall also denote by $\norm[X]{\cdot}$ the $L^2$-norm.

Let \(\calH\subset(0, \infty)\) be a countable set of mesh sizes with a unique cluster point at \(0\). For each \(h\in\calH\), we partition the domain \(\Omega\) into a mesh \(\Mh=(\Th, \Fh)\), where $\Th$ denotes the mesh elements and $\Fh$ the mesh faces.

We suppose that the mesh elements $\Th$ are a disjoint set of bounded simply connected domains in $\bbR^d$ with piece-wise $C^1$ boundary $\partial T$. We further suppose that $\ol{\Omega} = \bigcup_{T\in\Th} \ol{T}$.

We suppose that the mesh faces $\Fh$ are a disjoint set of non-intersecting, finite, $(d-1)$-dimensional $C^1$ manifolds which partition the mesh skeleton: $\bigcup_{T\in\Th} \partial T = \bigcup_{F\in\Fh} \ol{F}$, and that for each $F\in\Fh$ there either exists two distinct elements $T_1,T_2\in\Th$ such that $F \subset \partial T_1 \cap \partial T_2$ and $F$ is called an internal face, or there exists one element $T\in\Th$ such that $F \subset \partial T \cap \partial \Omega$ and $F$ is called a boundary face. Interior faces are collected in the set $\Fh^i$ and boundary faces in the set $\Fh^b$.

The parameter \(h\) is given by \(h\defeq\max_{T\in\Th}\hT\) where, for \(X=T\in\Th\) or \(X=F\in\Fh\), \(\hX\) denotes the diameter of \(X\). We shall also collect the faces attached to an element \(T\in\Th\) in the set \(\FT:=\{F\in\Fh:F\subset T\}\). The unit normal to \(F\in\FT\) pointing outside \(T\) is denoted by \(\norTF\), and \(\norT:\bdryT\to\R^d\) is the unit normal defined by \((\norT)|_F=\norTF\) for all \(F\in\FT\). We note that as each $F$ is $C^1$, the normal $\norTF$ is well defined. It is also worth noting that the normal vector $\norTF$ will not be constant on curved faces.

We consider the following regularity assumption on the mesh elements.

\begin{assumption}[Regular mesh sequence]\label{assum:star.shaped}	
	There exists a constant \(\varrho>0\) such that, for each \(h\in\calH\), every \(T\in\Th\) is connected by star-shaped sets with parameter \(\varrho\) (see \cite[Definition 1.41]{di-pietro.droniou:2020:hybrid}). 
\end{assumption}

\begin{remark}[Assumptions on the mesh]
	Assumption \ref{assum:star.shaped} is taken from \cite[Assumption 1]{droniou.yemm:2022:robust} however we have removed the assumption that the faces are connected by star shaped sets. We shall also note that there is no requirement that the mesh elements be polytopal or for the mesh faces to be planar. 
\end{remark}

We further require that the elements of the mesh align with the discontinuities of the diffusion tensor, i.e., for each \(T\in\Th\), \(\matK|_T\defeq \matKT\) is a constant matrix. In an analogous manner to \eqref{eq:uniform.elliptic} we define quantities \(0 < \ulKT \le \olKT\) to satisfy
\begin{equation}\label{eq:local.uniform.elliptic}
\ulKT\bmx\cdot\bmx \le (\matKT\bmx )\cdot\bmx \le \olKT\bmx\cdot\bmx \qquad\forall\bmx\in\R^d,
\end{equation}
and we define the local diffusion anisotropy ratio \(\alphaT \defeq \frac{\olKT}{\ulKT}\).

From hereon we shall write $f \lesssim g$ if there exists some constant $C$ which is independent of the quantities $f$ and $g$, the mesh diameter $h$, and of the diffusion tensor $\matK$, such that $f \le C g$.

Under Assumption \ref{assum:star.shaped} the following continuous trace inequality holds: for all $v\in \HONE(T)$,
\begin{equation}\label{eq:continuous.trace}
	\hT\norm[\bdryT]{v}^2 \lesssim \norm[T]{v}^2+\hT^2\norm[T]{\nabla v}^2. 
\end{equation}
A proof of \eqref{eq:continuous.trace} is provided in \cite{droniou.yemm:2022:robust}. We note that no assumption on $T$ being polytopal is required. We also note the following inverse Sobolev inequality, a proof of which is provided for highly generic and potentially curved elements in \cite[Lemma 4.23]{cangiani.dong.ea:2022:hp}:
\begin{equation}\label{eq:inverse.sobolev}
	\norm[T]{\nabla v}^2 \lesssim \hT^{-2}\norm[T]{v}^2 \qquad \forall v\in\POLY{\ell}(T),
\end{equation}
where we denote by $\POLY{\ell}(T)$ the space of polynomials on $T$ of degree $\le \ell$, $\ell\in\bbN$. Combining \eqref{eq:continuous.trace} and \eqref{eq:inverse.sobolev} yields the following discrete trace inequality:
\begin{equation}\label{eq:discrete.trace}
	\hT\norm[\bdryT]{v}^2 \lesssim \norm[T]{v}^2 \qquad \forall v\in\POLY{\ell}(T).
\end{equation}


\section{Discrete Model}\label{sec:discrete.problem}

A standard hybrid high-order method on polytopal meshes defines the local discrete space as
\[
	\UTk = \POLY{k}(T) \bigtimes_{F\in\Fh[T]} \POLY{k}(F).
\]
This makes sense on polytopal meshes where $F$ is a $(d-1)$-dimensional hyperplane as there is no ambiguity in what is meant by $\POLY{k}(F)$. Indeed, on such meshes it holds that $\POLY{k}(F) = \POLY{k}(\Omega)|_F = \POLY{k}(\Omega)^d \cdot \norF$. On curved meshes, it is not so obvious what the discrete space should be.

We find that the appropriate local discrete space is that of
\begin{equation}\label{eq:local.discrete.space.def}
	\UTk = \POLY{k}(T) \bigtimes_{F\in\Fh[T]} \calP^k(F),
\end{equation}
where we define
\begin{equation}\label{eq:curved.face.space}
	\calP^k(F) \defeq \POLY{0}(F) + \POLY{k}(\Omega)^d \cdot \norF,
\end{equation}
and $\norF$ is an arbitrary unit normal to the face $F$. The choice of unit normal $\norF$ does not affect the definition of $\calP^k(F)$.
We note that, even for a curved face, there is no ambiguity in the term $\POLY{0}(F)$ as it represents the space of functions which are constant on the face $F$. We emphasise that as the unit normal $\norF$ is not constant, the space $\calP^k(F)$ will be non-polynomial on curved faces.

\begin{remark}
	If $F$ is planar (that is, a $(d-1)$-dimensional hyperplane) then it holds that $\calP^k(F) = \POLY{k}(F)$ and thus the discrete space in \eqref{eq:local.discrete.space.def} coincides with the usual HHO space.
\end{remark}

\begin{remark}
	It suffices to take the space of unknowns on the faces as $\POLY{0}(F) + (\matK_{T_1}\nabla\POLY{k+1}(\Omega)) \cdot \nor_{F} + (\matK_{T_2}\nabla\POLY{k+1}(\Omega)) \cdot \nor_{F}$ with $\{T_1,T_2\}=\Th[F]$ for stability and consistency to hold. However, we define the space as $\calP^k(F) = \POLY{0}(F) + \POLY{k}(\Omega)^d \cdot \norF$ for simpler implementation and robustness of more general models.
\end{remark}

We shall denote by $\calP^{k}(\Fh[T])$ the space
\begin{equation*}
	\calP^{k}(\Fh[T]) \defeq \{v\in L^1(\partial T) : v|_F \in \calP^{k}(F)\ \forall F\in\Fh[T]\},
\end{equation*}
and for a given $\ulvT\in\UTk$ we shall write $\ulvT=(\vT,\vFT)$ with $\vT\in\POLY{k}(T)$ and $\vFT\in\calP^k(\Fh[T])$. The potential reconstruction $\pKT{k+1}:\UTk\to\POLY{k+1}(T)$ is defined as the unique solution to
\begin{subequations}\label{eq:pT}
	\begin{align}
		\brac[T]{\nabla \pKT{k+1} \ulvT, \nabla w} \eq -\brac[T]{\vT, \nabla\cdot(\matKT\nabla w)} + \brac[\partial T]{\vFT, (\matKT\nabla w) \cdot \norT} \quad \forall w \in \POLY{k+1}(T), \label{eq:pT.def}\\
		\int_T(\pKT{k+1} \ulvT - \vT) \eq 0 \label{eq:pT.closure}.
	\end{align}
\end{subequations} 
We denote by $\piTzr{k}:L^1(T) \to \POLY{k}(T)$ and $\piFzr{k}:L^1(F) \to \calP^{k}(F)$ the $L^2$-orthogonal projectors onto the spaces $\POLY{k}(T)$ and $\calP^{k}(F)$ respectively. We denote by $\piKTe{k + 1}:L^1(T) \to \POLY{k+1}(T)$ the oblique elliptic projector onto the space $\POLY{k+1}(T)$ satisfying
\begin{subequations}\label{eq:piTe}
	\begin{align}
	\brac[T]{\matKT\nabla (v - \piKTe{k+1}v), \nabla w} \eq 0\quad \forall w \in \POLY{k+1}(T), \label{eq:piTe.def}\\
	\int_T(\piKTe{k+1} v - v) \eq 0 \label{eq:piTe.closure}.
\end{align}
\end{subequations} 

The following weighted inner-products and norms are taken from \cite{droniou.yemm:2022:robust}. The weighted inner-product \(\brac[\matK, \bdryT]{\cdot, \cdot}:\LTWO(\bdryT)\times\LTWO(\bdryT)\to\R\) is defined for all \(v,w\in \LTWO(\bdryT)\) via
\begin{equation}\label{eq:ip.weighted.bdry}
\brac[\matK, \bdryT]{v, w} \defeq \brac[\bdryT]{\matKT^{\frac12}\norT\,v,\matKT^{\frac12}\norT\,w}=\brac[\bdryT]{[\matKT\norT\cdot\norT]v, w}.
\end{equation}
For all \(r\ge 1\) and \(v\in \HS{r}(T)\) the weighted \(\HS{r}\)-seminorm \(\seminorm[\matK, \HS{r}(T)]{{\cdot}}\) is defined as
\begin{equation}\label{eq:norm.weighted.hr}
\seminorm[\matK, \HS{r}(T)]{v} \defeq \seminorm[H^{r-1}(T)^d]{\matKT^{\frac12}\nabla v}. 
\end{equation}
\begin{lemma}[Approximation properties of $\piKTe{k+1}$]
	For all $s= 1, \dots, k+1$ and $v\in\HS{k + 2}(T)$,
	\begin{equation}\label{eq:piTe.approx}
		\seminorm[\matK,\HS{s}(T)]{v - \piKTe{k+1}v} \lesssim \hT^{k + 2 - s}\seminorm[\matK,\HS{k + 2}(T)]{v}.
	\end{equation}
\end{lemma}
\begin{proof}
	A proof is provided by \cite[Lemma 9]{droniou.yemm:2022:robust}. While that particular proof assumes the elements are polytopal, the proof only relies on \cite[Theorem 1.50]{di-pietro.droniou:2020:hybrid} which is provided for generic elements connected by star-shaped sets.
\end{proof}
The interpolant $\ITk:\HONE(T)\to\UTk$ is defined by 
\begin{equation}\label{eq:interpolant}
	\ITk v = (\piTzr{k} v, \piFTzr{k} v),
\end{equation}
where $\piFTzr{k}:\LP{1}(T)\to\calP^{k}(\Fh[T])$ is defined such that $\piFTzr{k}|_F=\piFzr{k}$ for all $F\in\Fh[T]$.

\begin{lemma}\label{lem:commutation}
	The following commutation property holds:
	\begin{equation}\label{eq:commutation}
		\pKT{k+1} \circ\ITk = \piKTe{k+1} .
	\end{equation} 
\end{lemma}

\begin{proof}
	It follows from the definitions of $\pKT{k+1}$ and $\ITk$ that
	\[
		\brac[T]{\nabla \pKT{k+1}\ITk v, \nabla w} = -\brac[T]{\piTzr{k} v, \nabla\cdot(\matKT\nabla w)} + \brac[\partial T]{\piFTzr{k} v, (\matKT\nabla w) \cdot \norT}\qquad \forall w \in \POLY{k+1}(T).
	\]
	However, as $\nabla\cdot(\matKT\nabla w) \in \POLY{k}(T)$ and $(\matKT\nabla w) \cdot \norT \in (\matKT\nabla\POLY{k+1}(T)) \cdot \norT \subset \calP^{k}(\Fh[T])$ each of the projectors $\piTzr{k}$ and $\piFTzr{k}$ can be removed to yield
	\begin{align*}
		\brac[T]{\nabla \pKT{k+1}\ITk v, \nabla w} \eq -\brac[T]{v, \nabla\cdot(\matKT\nabla w)} + \brac[\partial T]{v, (\matKT\nabla w) \cdot \norT} \\ \eq \brac[T]{\nabla v, \nabla w} = \brac[T]{\nabla \piKTe{k+1}v, \nabla w},
	\end{align*}
	where in the last two equalities we have integrated by parts and introduced the oblique elliptic projector using equation \eqref{eq:piTe.def}. Taking $w = \pKT{k+1}\ITk v - \piKTe{k+1}v$ we observe that
	\[
		\norm[T]{\nabla(\pKT{k+1}\ITk v - \piKTe{k+1}v)}^2=0.
	\]
	Combining with $\int_T (\pKT{k+1}\ITk v - \piKTe{k+1}v) = 0$ (due to equations \eqref{eq:pT.closure} and \eqref{eq:piTe.closure}) we conclude that
	\[
		\pKT{k+1}\ITk v - \piKTe{k+1}v = 0.
	\]
\end{proof}

\begin{remark}
	We note that the commutation property \eqref{eq:commutation} is the key result required to prove consistency of the scheme and relies on the fact that $(\matKT\nabla\POLY{k+1}(T)) \cdot \norTF \subset \calP^{k}(F)$ for each $F$. The additional condition that $\POLY{0}(F) \subset \calP^{k}(F)$ is required for coercivity to hold.
\end{remark}
We endow the discrete space $\UTk$ with the seminorm
\begin{equation}\label{eq:discrete.norm.def}
	\norm[1,\matK,T]{\ulvT}^2 = \seminorm[\matK,\HONE(T)]{\vT}^2 + \hT^{-1}\norm[\matK, \partial T]{\vFT - \vT}^2.
\end{equation}
The local bilinear form \(\aKT:\UTk\times\UTk\to\bbR\) is defined as
\begin{equation}\label{eq:local.form.def}
\aKT(\uluT,\ulvT) \defeq \brac[T]{\matKT\nabla \pKT{k+1}\uluT, \nabla \pKT{k+1}\ulvT} + \sKT(\uluT, \ulvT), 
\end{equation}
where \(\sKT:\UTk\times\UTk\to\R\) is a local stabilisation term such that the following assumptions hold.
\begin{assumption}[Local stabilisation term]\label{assum:stability}
	The stabilisation term \(\sKT\) is a symmetric, positive semi-definite bilinear form that satisfies:
	\begin{enumerate}
		\item \emph{Stability and boundedness.} For all \(\ulvT\in\UTk\),
		\begin{equation}\label{eq:sT.stability.and.boundedness}
			\alphaT^{-1}\norm[1,\matK,T]{\ulvT}^2 \lesssim \aKT(\ulvT,\ulvT) \lesssim \alphaT\norm[1,\matK,T]{\ulvT}^2.
		\end{equation}
		\item \emph{Polynomial consistency}. For all $\ulvT\in\UTk$ and \(w\in\POLY{k+1}(T)\),
		\begin{equation}\label{eq:sT.polynomial.consistency}
			\sKT(\ulvT, \ITk w)  = 0.
		\end{equation}
	\end{enumerate}
\end{assumption}

An example of a stabilisation bilinear form satisfying Assumption \ref{assum:stability} is provided in Section \ref{sec:stabilisation}.

\begin{lemma}[Consistency of $\sKT$]
	Suppose \(\sKT:\UTk\times\UTk\to\R\) satisfies Assumption \ref{assum:stability}. Then it holds for all $w\in\HS{k+2}(T)$ that
	\begin{equation}\label{eq:sT.consistency}
		\sKT(\ITk w, \ITk w) \lesssim \SqBrac{\alphaT\hT^{k+1}\seminorm[\matK,\HS{k+2}(T)]{w}}^2.
	\end{equation}
\end{lemma}

\begin{proof}
	It follows from equation \eqref{eq:sT.polynomial.consistency} that
	\begin{equation}
		\sKT(\ITk w, \ITk w) = \sKT(\ITk (w - \piKTe{k+1}w), \ITk (w - \piKTe{k+1}w)).
	\end{equation}
	Therefore, applying the upper bound in \eqref{eq:sT.stability.and.boundedness} and the definition \eqref{eq:discrete.norm.def} of $\norm[1,\matK,T]{\cdot}$ yields
	\begin{multline*}
		\sKT(\ITk w, \ITk w) \lesssim \\ \alphaT\Brac{\seminorm[\matK,\HONE(T)]{\piTzr{k}(w - \piKTe{k+1}w)}^2 + \hT^{-1} \norm[\matK,\bdryT]{\piFTzr{k}(w - \piKTe{k+1}w) - \piTzr{k}(w - \piKTe{k+1}w)}^2}.
	\end{multline*}
	Thus, we infer from Lemma \ref{lem:h1.bound.l2.projectors} below that
	\begin{equation*}
		\sKT(\ITk w, \ITk w) \lesssim \alphaT^2\seminorm[\matK,\HONE(T)]{w - \piKTe{k+1}w}^2.
	\end{equation*}
	The proof follows from the approximation properties \eqref{eq:piTe.approx} of $\piKTe{k+1}$.
\end{proof}

\subsection{Global Space and HHO Scheme}

The global space of unknowns is defined as 
\begin{equation}\label{eq:global.space.def}
	\Uhk\defeq \bigtimes_{T\in\Th} \POLY{k}(T)\ \times \bigtimes_{F\in\Fh} \calP^k(F).
\end{equation}
To account for the homogeneous boundary conditions, the following subspace is also introduced,
\begin{equation}\label{eq:global.space.hom.def}
	\Uhkzr\defeq\{\ulvh \in\Uhk:v_{F}=0\quad\forall F\in\Fhb\}.
\end{equation}
For any \(\ulvh\in\Uhk\) we denote its restriction to an element \(T\) by \(\ulvT=(\vT,\vFT)\in\UTk\) (where, naturally, \(\vFT\) is defined form \((\vF)_{F\in\FT}\)). We also denote by \(\vh\) the piecewise polynomial function satisfying \(\vh|_T=\vT\) for all \(T\in\Th\).
The global bilinear forms \(\aKh:\Uhk\times\Uhk\to\bbR\) and \(\sKh:\Uhk\times\Uhk\to\mathbb{R}\) are defined as
\[
\aKh(\uluh, \ulvh) \defeq \sum_{T\in\Th} \aKT(\uluT,\ulvT)
\quad\textrm{and}\quad
\sKh(\uluh, \ulvh) \defeq \sum_{T\in\Th} \sKT(\uluT,\ulvT).
\]
The HHO scheme reads: find \(\uluh\in\Uhkzr\) such that
\begin{equation}\label{eq:discrete.problem}
	\aKh(\uluh, \ulvh) = \ell_h(\ulvh) \qquad\forall \ulvh\in\Uhkzr, 
\end{equation}
where \(\ell_h:\Uhkzr\to\R\) is a linear form defined as
\begin{equation}\label{eq:discrete.src.term}
	\ell_h(\ulvh) \defeq \sum_{T\in\Th}\brac[T]{f,\vT}. 
\end{equation}

We define the discrete energy norm \(\norm[\rma, \matK, h]{{\cdot}}\) on \(\Uhkzr\) as
\begin{equation}\label{eq:energy.norm.def}
\norm[\rma, \matK, h]{\ulvh}\defeq \aKh(\ulvh,\ulvh)^\frac{1}{2} \qquad \forall \ulvh\in\Uhk.
\end{equation}	

\begin{lemma}\label{lem:ah.norm}
The mapping \(\norm[\rma, \matK, h]{{\cdot}}:\Uhkzr\to\R\) defines a norm on $\Uhkzr$.
\end{lemma}

\begin{proof}
As $\norm[\rma, \matK, h]{{\cdot}}$ is clearly a seminorm we only need to prove that if $\norm[\rma, \matK, h]{{\ulvh}} = 0$ then $\ulvh = 0$. It follows from the boundedness \eqref{eq:sT.boundedness} that
\[
\sum_{T\in\Th} \SqBrac{\seminorm[\matK,\HONE(T)]{\vT}^2 + \hT^{-1}\norm[\matK, \partial T]{\vFT - \vT}^2} \lesssim \norm[\rma, \matK, h]{{\ulvh}}^2.
\]
Thus, if $\norm[\rma, \matK, h]{{\ulvh}} = 0$ then it must hold that $\vT = \vF = {\rm{const}}$ for every $T\in\Th$, $F\in\Th$. However, we infer from the homogeneous boundary conditions that those constants must all be zero.
\end{proof}

\section{Error estimates}\label{sec:error}

\begin{theorem}[Consistency error]\label{thm:consistency}
	The consistency error \(\calE_h( w;\cdot) :\Uhkzr\to\R\) is the linear form defined for all \(\ulvh \in\Uhkzr\) as
	\[
	\calE_h(w; \ulvh) \defeq -\brac[\Omega]{\nabla\cdot(\matKT\nabla w), \vh} - \aKh(\Ihk w, \ulvh),
	\]
	for any \(w\in \HONEzr( \Omega) \) such that \(\nabla\cdot(\matKT\nabla w)\in \LTWO(\Omega)\). If such a \(w\) additionally satisfies \(w|_T\in\HS{k+2}(T)\) for all $T\in\Th$, the consistency error satisfies
	\begin{equation}\label{eq:cons.error}
		|\calE_h(w; \ulvh)| \lesssim
		\errorRHS{w}\norm[\rma, \matK, h]{\ulvh}.
	\end{equation}
\end{theorem}

The global operators \(\pKh{k+1}:\Uhk\to\POLY{k+1}(\Th)\) and \(\piKhe{k+1}:\HONE(\Th)\to\POLY{k+1}(\Th)\) are defined such that their actions restricted to an element \(T\in\Th\) are that of \(\pKT{k+1}\) and \(\piKTe{k+1}\). The global interpolator \(\Ihk:\HONE(\Omega)\to\Uhk\) is defined as \(\Ihk v \defeq ((\piTzr{k}v)_{T\in\Th},(\piFzr{k}v)_{F\in\Fh})\). 

\begin{theorem}[Energy and $L^2$ error estimates]\label{thm:error}
	Let $u\in\HONEzr(\Omega)$ be the exact solution to equation \eqref{eq:weak.form} and suppose the additional regularity $u \in \HS{k+2}(\Th)$. Let $\uluh$ be the exact solution to the discrete problem \eqref{eq:discrete.problem}. Then the following error estimates hold:
	\begin{itemize}
		\item \emph{Energy estimate}.
		\begin{equation}\label{eq:energy.error}
			\norm[\rma, \matK, h]{\uluh - \Ihk u} + \seminorm[\HONE(\Th)]{\pKh{k+1}\uluh - u} \lesssim \errorRHS{u}.
		\end{equation}
		\item \emph{$L^2$ estimate}. Suppose additionally that the domain $\Omega$ is convex and $\matK=\mat{I}$ is the identity matrix, then optimal convergence in $L^2$-norm holds:
		\begin{equation}\label{eq:L2.error}
			\norm[\Omega]{\pKh{k+1}\uluh - u} \lesssim h^{k+2}\seminorm[\HS{k+2}(\Th)]{u},
		\end{equation}
		where the seminorm $\seminorm[\HS{s}(\Th)]{\cdot}$ is defined as the square-root of the sum of squares of $\seminorm[\HS{s}(T)]{\cdot}$ for any $s\in\bbN$.
	\end{itemize}
\end{theorem}

\begin{remark}
	The $L^2$-error estimate is stated with identity diffusion, corresponding to a Poisson problem. However, the result follows trivially (with a hidden constant depending additionally on the anisotropy of $\matK$) for any constant diffusion tensor $\matK$ \cite[cf.][Remark 3.21]{di-pietro.droniou:2020:hybrid}.
\end{remark}

\begin{proof}[Proof of Theorems \ref{thm:consistency} and \ref{thm:error}]
The estimates \eqref{eq:cons.error} and \eqref{eq:energy.error} are provided in \cite{droniou.yemm:2022:robust} and rely only on the design conditions stated in Assumption \ref{assum:stability}, the commutation property \eqref{eq:commutation}, the approximation properties of the elliptic projector \eqref{eq:piTe.approx}, the consistency of $\sKT$ \eqref{eq:sT.consistency}, Lemma \ref{lem:ah.norm}, as well as standard trace and inverse estimates provided in Section \ref{sec:model.and.mesh}.

To prove \eqref{eq:L2.error} we require a slightly different approach to that of \cite[Theorem 2.32]{di-pietro.droniou:2020:hybrid}. In particular, as $\piFzr{k}$ is not a polynomial projector \cite[Equation (2.78)]{di-pietro.droniou:2020:hybrid} does not hold in our case. However, the remainder of the proof is the same so we only have to show that 
\begin{equation}
	\sup_{g\in\LTWO(\Omega):\norm[\Omega]{g}\le 1}|\calE_h(u; \Ihk z_g)| \lesssim h^{k+2}\seminorm[\HS{k+2}(\Th)]{u},
\end{equation}
where $z_g$ is the solution to the dual problem 
\[
	\a(v, z_g) = \brac[\Omega]{g, v} \qquad \forall v\in\HONEzr(\Omega).
\]
As we have assumed $\Omega$ to be convex, the following elliptic regularity holds:
\begin{equation}\label{eq:elliptic.regularity}
	\norm[\HS{2}(\Omega)]{z_g} \lesssim \norm[\Omega]{g}.
\end{equation}
Moreover, as $\matK=\mat{I}$, the following equality established in the proof of \cite[Lemma 2.18]{di-pietro.droniou:2020:hybrid} holds true:
\begin{equation}\label{eq:l2.proof.1}
	\calE_h(u; \Ihk z_g) = \sum_{T\in\Th} \Brac{\brac[\partial T]{\nabla (u - \piKTe{k+1}u) \cdot \norT, \piFTzr{k} z_g - \piTzr{k} z_g} - \sKT(\ITk u, \ITk z_g)}.
\end{equation}
The sum over the boundary term in \eqref{eq:l2.proof.1} can be written as follows,
\begin{multline*}
	\sum_{T\in\Th} \brac[\partial T]{\nabla (u - \piKTe{k+1}u) \cdot \norT, \piFTzr{k} z_g} \\ = \sum_{T\in\Th}\sum_{F\in\FT}\brac[F]{\nabla u \cdot \norTF, \piFzr{k} z_g} + \sum_{T\in\Th}\sum_{F\in\FT}\brac[F]{\nabla \piKTe{k+1}u \cdot \norTF, \piFzr{k} z_g}.
\end{multline*}
As $\nabla \piKTe{k+1}u \cdot \norTF\in\calP^{k}(F)$ we may drop the projector $\piFzr{k}$ to write
\[
	\sum_{T\in\Th}\sum_{F\in\FT}\brac[F]{\nabla \piKTe{k+1}u \cdot \norTF, \piFzr{k} z_g} = \sum_{T\in\Th}\sum_{F\in\FT}\brac[F]{\nabla \piKTe{k+1}u \cdot \norTF, z_g}
\]
As $\nabla u \in \bmH(\DIV;\Omega)$, the fluxes of $u$ are continuous across every internal face $F\in\Fhi$. Therefore, as $\piFzr{k} z_g = 0$ for all $F\in\Fhb$ (due to $z_g = 0$ on $\partial \Omega$), it holds that
\[
	\sum_{T\in\Th}\sum_{F\in\FT}\brac[F]{\nabla u \cdot \norTF, \piFzr{k} z_g} = 0 = \sum_{T\in\Th}\sum_{F\in\FT}\brac[F]{\nabla u \cdot \norTF, z_g}.
\]
Substituting back into \eqref{eq:l2.proof.1} yields
\begin{equation*}
	\calE_h(u; \Ihk z_g) = \sum_{T\in\Th} \Brac{\brac[\partial T]{\nabla (u - \piKTe{k+1}u) \cdot \norT, z_g - \piTzr{k} z_g} - \sKT(\ITk u, \ITk z_g)}.
\end{equation*}
It follows from a Cauchy--Schwarz inequality and the consistency \eqref{eq:sT.consistency} that
\[
	\sKT(\ITk u, \ITk z_g) \le \sKT(\ITk u, \ITk u)^\frac12 \sKT(\ITk z_g, \ITk z_g)^\frac12 \lesssim \hT^{k+1}\seminorm[\matK,\HS{k+2}(T)]{u} \hT \seminorm[\matK,\HS{2}(T)]{z_g}.
\]
It also follows from a Cauchy--Schwarz inequality, the continuous trace inequality \eqref{eq:continuous.trace} and the approximation properties \eqref{eq:piTe.approx} that
\begin{multline*}
	\brac[\partial T]{\nabla (u - \piKTe{k+1}u) \cdot \norT, z_g - \piTzr{k} z_g} \le \norm[\partial T]{\nabla (u - \piKTe{k+1}u)\cdot\norT} \norm[\partial T]{z_g - \piTzr{k} z_g} \\ \lesssim \hT^{k+1}\seminorm[\HS{k+2}(T)]{u}\hT^{-\frac12}\norm[\partial T]{z_g - \piTzr{k} z_g}.
\end{multline*}
Thus, we need to prove that 
\begin{equation*}
	\hT^{-\frac12}\norm[\partial T]{z_g - \piTzr{k} z_g} \lesssim \hT \seminorm[\HS{2}(T)]{z_g}
\end{equation*}
and the proof follows from the elliptic regularity \eqref{eq:elliptic.regularity} and the bound $\norm[\Omega]{g}\le 1$. By a continuous trace inequality and a Poincar\'{e}--Wirtinger inequality
\[
	\hT^{-\frac12}\norm[\partial T]{z_g - \piTzr{k} z_g} \lesssim \norm[T]{\nabla(z_g - \piTzr{k} z_g)}.
\]
The result holds due to the $H^1$-approximation properties of the $L^2$-projector \cite[Lemma 1.43]{di-pietro.droniou:2020:hybrid} which remain valid in curved domains.
\end{proof}

\section{Analysis of the stabilisation}\label{sec:stabilisation}

We consider here the stabilisation bilinear form defined by
\begin{multline}\label{eq:sT.def}
	\sKT(\ulvT, \ulwT) = \brac[T]{\matKT\nabla (\vT - \piTzr{k}\pKT{k+1}\ulvT), \nabla (\wT - \piTzr{k}\pKT{k+1}\ulwT)} \\
	+ \hT^{-1}\brac[\matK,\bdryT]{\vFT - \piFTzr{k}\pKT{k+1}\ulvT, \wFT - \piFTzr{k}\pKT{k+1}\ulwT},
\end{multline}
however, the arguments we use to show robustness on curved meshes extend seamlessly to more general choices of stability such as those considered in \cite[Section 4]{droniou.yemm:2022:robust}. It is clear that $\sKT$ satisfies \eqref{eq:sT.polynomial.consistency} so it remains to prove that \eqref{eq:sT.stability.and.boundedness} holds.

\begin{lemma}\label{lem:h1.bound.l2.projectors}
	It holds for all $v\in\HONE(T)$ that
	\begin{equation}\label{eq:h1.bound.l2.projectors}
	\seminorm[\matK,\HONE(T)]{\piTzr{k} v}^2 + \hT^{-1}\norm[\matK, \partial T]{\piFTzr{k} v - \piTzr{k} v}^2 \lesssim \alphaT\seminorm[\matK,\HONE(T)]{v}^2.
	\end{equation}
\end{lemma}

\begin{proof}
	We first note the bound
	\[
		\seminorm[\matK,\HONE(T)]{\piTzr{k} v}^2 + \hT^{-1}\norm[\matK, \partial T]{\piFTzr{k} v - \piTzr{k} v}^2 \le \olKT \Brac{\seminorm[\HONE(T)]{\piTzr{k} v}^2 + \hT^{-1}\norm[\partial T]{\piFTzr{k} v - \piTzr{k} v}^2}
	\]
	which follows from the ellipticity \eqref{eq:local.uniform.elliptic} of $\matKT$. Consider, by a triangle inequality
	\[
		\hT^{-1}\norm[\partial T]{\piFTzr{k} v - \piTzr{k} v}^2 \lesssim \hT^{-1}\norm[\partial T]{v - \piFTzr{k} v}^2 + \hT^{-1}\norm[\partial T]{v - \piTzr{k} v}^2.
	\]
	First, we wish to bound the term $\norm[\partial T]{v - \piFTzr{k} v}$. As $\piFTzr{k}$ is the $L^2$-orthogonal projector on $\calP^{k}(\Fh[T])$, it minimises its respective norm. Therefore, we may replace $\piFTzr{k} v$ with any element of $\calP^{k}(\Fh[T])$. In particular, as $\POLY{0}(T)|_{\partial T} \subset \calP^{k}(\Fh[T])$ it holds that
	\[
		\hT^{-1}\norm[\partial T]{v - \piFTzr{k} v}^2 \le \hT^{-1}\norm[\partial T]{v - \piTzr{0} v}^2.
	\]
	It follows from the continuous trace inequality \eqref{eq:continuous.trace} and a Poincar\'{e}--Wirtinger inequality that
	\[
		\hT^{-1}\norm[\partial T]{v - \piTzr{0} v}^2 \lesssim \norm[T]{\nabla v}^2.
	\]
	Similarly, we apply the continuous trace inequality and a Poincar\'{e}--Wirtinger inequality on the term $\hT^{-1}\norm[\partial T]{v - \piTzr{k} v}^2$ to yield
	\[
		\hT^{-1}\norm[\partial T]{v - \piTzr{k} v}^2 \lesssim \norm[T]{\nabla(v - \piTzr{k} v)}^2 \lesssim \norm[T]{\nabla v}^2 + \norm[T]{\nabla \piTzr{k} v}^2,
	\]
	where we have applied a triangle inequality to reach the conclusion. It follows from \cite[Equation (1.77)]{di-pietro.droniou:2020:hybrid} (which invokes \cite[Equation (1.74)]{di-pietro.droniou:2020:hybrid} which does not rely on the elements being polytopal) that
	\[
		\norm[T]{\nabla \piTzr{k} v}^2 \lesssim \norm[T]{\nabla v}^2.
	\]
	Thus, we can conclude that
	\[
	\seminorm[\matK,\HONE(T)]{\piTzr{k} v}^2 + \hT^{-1}\norm[\matK, \partial T]{\piFTzr{k} v - \piTzr{k} v}^2 \le \olKT \norm[T]{\nabla v}^2.
	\]
	The proof follows by applying the ellipticity \eqref{eq:local.uniform.elliptic} of $\matKT$ to yield 
	\[
		\norm[T]{\nabla v}^2 \lesssim \ulKT^{-1}\seminorm[\matK,\HONE(T)]{v}^2.
	\]
\end{proof}

\begin{remark}
	We note that the inclusion $\POLY{0}(F)\subset \calP^{k}(F)$ is crucial for the bound
	\[
		\hT^{-1}\norm[\matK, \partial T]{\piFTzr{k} v - \piTzr{k} v}^2 \lesssim \alphaT\seminorm[\matK,\HONE(T)]{v}^2
	\]
	to hold, and without this inclusion, coercivity cannot hold.
\end{remark}

\begin{lemma}[Coercivity]
	It holds for all $\ulvT\in\UTk$ that
	\begin{equation}\label{eq:sT.coercivity}
		\norm[1,\matK,T]{\ulvT}^2 \lesssim \aKT(\ulvT, \ulvT).
	\end{equation}
\end{lemma}

\begin{proof}
	It follows from the definition \eqref{eq:discrete.norm.def} of $\norm[1,\matK,T]{\cdot}$ that
	\begin{multline*}
		\norm[1,\matK,T]{\ulvT}^2 = \seminorm[\matK,\HONE(T)T]{\vT}^2 + \hT^{-1}\norm[\matK,\partial T]{\vFT - \vT}^2 \\ \lesssim \seminorm[\matK,\HONE(T)T]{\vT - \piTzr{k}\pKT{k+1}\ulvT}^2 + \seminorm[\matK,\HONE(T)T]{\piTzr{k}\pKT{k+1}\ulvT}^2+ \hT^{-1}\norm[\matK,\partial T]{\vFT - \piFTzr{k}\pKT{k+1}\ulvT}^2 \\  + \hT^{-1}\norm[\matK,\partial T]{\piFTzr{k}\pKT{k+1}\ulvT - \piTzr{k}\pKT{k+1}\ulvT}^2 + \hT^{-1}\norm[\matK,\partial T]{\piTzr{k}\pKT{k+1}\ulvT - \vT}^2,
	\end{multline*}
	where we have added and subtracted $\piTzr{k}\pKT{k+1}\ulvT$ to the volumetric term, and $\piTzr{k}\pKT{k+1}\ulvT$ and $\piFTzr{k}\pKT{k+1}\ulvT$ to the boundary term, and invoked triangle inequalities to reach the conclusion. Similar to the proof of Lemma \ref{lem:h1.bound.l2.projectors}, we apply the continuous trace inequality \eqref{eq:continuous.trace}, a Poincar\'{e}--Wirtinger inequality (due to the zero mean value of $\piTzr{k}\pKT{k+1}\ulvT - \vT$) and the ellipticity \eqref{eq:local.uniform.elliptic} of $\matKT$ to yield
	\[
		\hT^{-1}\norm[\matK, \partial T]{\piTzr{k}\pKT{k+1}\ulvT - \vT}^2 \lesssim \seminorm[\matK, \HONE(T)]{\vT - \piTzr{k}\pKT{k+1}\ulvT}^2.
	\]
	Therefore,
	\begin{align*}
		\norm[1,\matK,T]{\ulvT}^2 \les \alphaT\seminorm[\matK,\HONE(T)]{\vT - \piTzr{k}\pKT{k+1}\ulvT}^2 + \seminorm[\matK,\HONE(T)]{ \piTzr{k}\pKT{k+1}\ulvT}^2+ \hT^{-1}\norm[\matK,\partial T]{\vFT - \piFTzr{k}\pKT{k+1}\ulvT}^2 \\  \plus \hT^{-1}\norm[\matK,\partial T]{\piFTzr{k}\pKT{k+1}\ulvT - \piTzr{k}\pKT{k+1}\ulvT}^2 \\ \lea \alphaT\sKT(\ulvT, \ulvT) + \seminorm[\matK,\HONE(T)]{ \piTzr{k}\pKT{k+1}\ulvT}^2 + \hT^{-1}\norm[\matK,\partial T]{\piFTzr{k}\pKT{k+1}\ulvT - \piTzr{k}\pKT{k+1}\ulvT}^2.
	\end{align*}
	We can conclude from Lemma \ref{lem:h1.bound.l2.projectors} that
	\[
		\seminorm[\matK,\HONE(T)]{ \piTzr{k}\pKT{k+1}\ulvT}^2 + \hT^{-1}\norm[\matK,\partial T]{\piFTzr{k}\pKT{k+1}\ulvT - \piTzr{k}\pKT{k+1}\ulvT}^2 \lesssim \alphaT\seminorm[\matK,\HONE(T)]{ \pKT{k+1}\ulvT}^2,
	\]
	which combined with the definition of $\aKT$ yields the result.
\end{proof}

\begin{lemma}[Boundedness]
	It holds for all $\ulvT\in\UTk$ that
	\begin{equation}\label{eq:sT.boundedness}
		\aKT(\ulvT, \ulvT) \lesssim \alphaT\norm[1,\matK,T]{\ulvT}^2.
	\end{equation}
\end{lemma}

\begin{proof}
Consider by a triangle inequality and Lemma \ref{lem:h1.bound.l2.projectors}
\begin{align*}
	\hT^{-1}\norm[\matK,\partial T]{\vFT &{}-  \piFTzr{k}\pKT{k+1}\ulvT}^2 \nl \les \hT^{-1}\norm[\matK,\partial T]{\vFT - \vT} + \hT^{-1}\norm[\matK,\partial T]{\vT - \pKT{k+1}\ulvT}^2 + \hT^{-1}\norm[\matK,\partial T]{\pKT{k+1}\ulvT - \piFTzr{k}\pKT{k+1}\ulvT}^2 \nl
	\les \hT^{-1}\norm[\matK,\partial T]{\vFT - \vT} + \alphaT\seminorm[\matK,\HONE(T)]{\vT - \pKT{k+1}\ulvT}^2 + \seminorm[\matK,\HONE(T)]{\pKT{k+1}\ulvT}^2 \nl
	\les \alphaT\norm[1,\matK,T]{\ulvT}^2 + \alphaT\seminorm[\matK,\HONE(T)]{\pKT{k+1}\ulvT}^2
\end{align*}
Similarly, by a triangle inequality and Lemma \ref{lem:h1.bound.l2.projectors},
\begin{equation*}
	\seminorm[\matK,\HONE(T)]{\vT - \piTzr{k}\pKT{k+1}\ulvT}^2 \lesssim \seminorm[\matK,\HONE(T)]{\vT}^2 + \seminorm[\matK,\HONE(T)]{\piTzr{k}\pKT{k+1}\ulvT}^2 \\ \lesssim \norm[1,\matK,T]{\ulvT}^2 + \alphaT\seminorm[\matK,\HONE(T)]{\pKT{k+1}\ulvT}^2.
\end{equation*}
Thus, we need to prove that
\[
	\seminorm[\matK,\HONE(T)]{\pKT{k+1}\ulvT}^2 \lesssim \norm[1,\matK,T]{\ulvT}^2.
\]
It follows from the definition \eqref{eq:pT} of $\pKT{k+1}$ and an integration by parts that
\begin{align}\label{eq:sT.bound.proof.1}
	\seminorm[\matK,\HONE(T)]{\pKT{k+1}\ulvT}^2 \eq \brac[T]{\nabla\vT, \matKT\nabla\pKT{k+1}\ulvT} + \brac[\partial T]{\vFT - \vT, \matKT\nabla\pKT{k+1}\ulvT \cdot \norT} \nl
	\eq \brac[T]{\matKT^\frac12\nabla\vT, \matKT^\frac12\nabla\pKT{k+1}\ulvT} + \brac[\partial T]{\matKT^\frac12 \norT (\vFT - \vT), \matKT^\frac12\nabla\pKT{k+1}\ulvT} \nl
	\les\seminorm[\matK,\HONE(T)]{\pKT{k+1}\ulvT}\Brac{\seminorm[\matK,\HONE(T)]{ \vT} + \hT^{-\frac12}\norm[\matK,\partial T]{\vFT - \vT}}.
\end{align}
where we have applied Cauchy--Schwarz inequalities on both inner-products and the discrete trace inequality \eqref{eq:discrete.trace}.
The proof follows by simplifying \eqref{eq:sT.bound.proof.1} by $\seminorm[\matK,\HONE(T)]{\pKT{k+1}\ulvT}$ and squaring.
\end{proof}

\section{Integration on curved domains}\label{sec:integration}

The design of integration methods on curved domains is an active area of research. In the recent article \cite{antolin.wei.ea:2022:robust} a quadrature rule for curved domains is developed by considering a decomposition into triangular or rectangular pyramids $\calT$ and a mapping $\bm{T}:[0,1]^d\to\calT$ for each decomposition. With knowledge of the Jacobian of such a mapping, integration can be performed on the pre-image of each $\calT$. The article \cite{chin:2021:scaled} develops an extension of the homogeneous integration rule developed in \cite{chin.lasserre.ea:2015:numerical} by considering a curved triangulation of the domain and constructing a scaled boundary parameterisation on each curved triangle.
Here, we also consider an extension of the homogeneous integration rule, but the approach we take is quite different. We avoid the need to split the curved domain into sub-regions and directly map the integral onto the boundary by constructing a Poincar\'{e}-type operator which inverts the divergence operator. Indeed, this operator was briefly mentioned in the appendix of \cite{chin:2021:scaled}, however, we develop the ideas here without a sub-triangulation, and independent of dimension.


We begin with the formula developed in \cite{chin.lasserre.ea:2015:numerical} to rewrite the integral onto the boundary of the element. This rule works by identifying a vector field
\begin{equation}\label{eq:div.inv.homogeneous}
	\bmF = \frac{\bmx v}{q+d}
\end{equation}
such that $\nabla \cdot \bmF = v$ for homogeneous functions $v$ of degree $q$. Therefore,
\begin{equation}\label{eq:integration.rule.homogeneous}
	\int_T v(\bmx)\dx = \int_{\partial T} \bmx\cdot\norT \frac{v(\bmx)}{q+d}\dS.
\end{equation}

We would like to extend this rule to non-homogeneous functions. We begin by searching for a vector field of the form 
\[
\bmF = 	g\hat{\bmr},
\]
such that $\nabla \cdot \bmF = v$ where $\hat{\bmr}$ denotes the unit vector in the radial direction. We find that the unknown function $g$ must satisfy
\[
\frac{1}{r^{d-1}}\frac{\partial}{\partial r}(r^{d-1} g) = v,
\]
where we denote by $r = |\bmx|$. A solution is given by
\[
g = \frac{1}{r^{d-1}} \int_{0}^r s^{d-1} v\big(\frac{s}{r}\bmx\big) \ds = r \int_{0}^1 t^{d-1} v(t \bmx) \dt.
\]
Thus, we have found an inverse divergence
\begin{equation}\label{eq:div.inv.from.zero}
	\bmF = r \hat{\bmr} \int_{0}^1 t^{d-1} v(t \bmx)\dt = \bmx \int_{0}^1 t^{d-1} v(t \bmx) \dt.
\end{equation}
Therefore, an integral over the element $T$ can be rewritten to its boundary as follows:
\begin{equation}\label{eq:integration.rule.from.zero}
	\int_T v(\bmx) \dx = \int_{\partial T} \bmx\cdot \norT \int_{0}^1 t^{d-1} v(t \bmx) \dt \dS.
\end{equation}
We note that if $v$ is a homogeneous function of degree $q$ (that is, $v(t \bmx) = t^q v(\bmx)$), then the inverse divergence formulae \eqref{eq:div.inv.homogeneous} and \eqref{eq:div.inv.from.zero} coincide and thus so do the rules \eqref{eq:integration.rule.homogeneous} and \eqref{eq:integration.rule.from.zero}. In this sense, the method can be considered an extension of the homogeneous integration rule developed in \cite{chin.lasserre.ea:2015:numerical}. 

If we instead consider a vector field of the form $\bmF = g\hat{\bmr}_0$ where $\hat{\bmr}_0$ is the unit radial direction from a shifted origin $\bmx_0$, we arrive at the more general formula
\begin{equation}\label{eq:div.inv.rule.general}
	\bmF = (\bmx - \bmx_0)\int_{0}^1 t^{d-1} v(t\bmx + (1-t)\bmx_0) \dt.
\end{equation}
Therefore, we may write
\begin{equation}\label{eq:integration.rule.general}
	\int_T v(\bmx) \dx = \int_{\partial T} (\bmx - \bmx_0) \cdot \norT \int_{0}^1 t^{d-1} v(t\bmx + (1-t)\bmx_0) \dt \dS.
\end{equation}

This is very useful if the element contains one or more planar faces. For a vertex with coordinates $\bmnu$, we can set $\bmx_0=\bmnu$ and it holds that $(\bmx - \bmnu) \cdot \norT = 0$ on any planar faces connected to the vertex $\bmnu$.  We note that if $T$ is not star-shaped with respect to $\bmnu$, then the integral $\int_{0}^1 t^{d-1} v(t\bmx + (1-t)\bmx_0) \dt$ will pass through points outside of $T$. Thus, one would require a sufficiently smooth extension of $v$ outside of $T$. However, for polynomials or functions analytic over $\Omega$ (such as an analytic source term), such an extension is trivial.


\subsection{A quadrature rule for curved edges in two dimensions}

For a given edge $E$, consider a parameterisation $\gamma_E : [t_0, t_1] \to E$, $t_0 < t_1$. Therefore, integration on curved edges is trivial:
\[
\int_E v(\bmx)\dE = \int_{t_0}^{t_1} v(\gamma_E(t)) |\gamma_E'(t)| \dt.
\]
The above integral can easily be approximated with a one-dimensional Gaussian quadrature rule. In particular, let $w_i$, $x_i$, $i=1,\dots,N$ be the weights and abscissae associated with a quadrature rule on $[0,1]$. Then we can generate weights $w_i^E$ and abscissae $\bmx_i^E$ on the edge $E$ as follows:
\begin{equation}\label{eq:weights.and.abscissae.edges}
	w_i^E = (t_1 - t_0) w_i |\gamma_E'(t_0 + (t_1 - t_0)x_i)|  \quad;\quad \bmx_i^E = \gamma_E(t_0 + (t_1 - t_0)x_i).
\end{equation}
In practise, we generally store an arc length parameterisation for each edge and thus the term $|\gamma_E'|$ is not required.

\subsection{A quadrature rule for elements in two dimensions}

In two dimensions the faces are edges and thus the boundary integral in \eqref{eq:integration.rule.general} can be evaluated on each edge $F\in\Fh[T]$ using the rule described in \eqref{eq:weights.and.abscissae.edges}. We let $w_i^F$ and $\bmx_i^F$, $i=1,\dots,N$ be the quadrature weights and abscissae associated with an edge $F\in\Fh[T]$ and $w_j$, $x_j$, $j=1,\dots,M$ be the weights and abscissae associated with a quadrature rule on $[0,1]$. We set $\bmnu$ to be the coordinate of a vertex of $T$ connected to the highest number of straight edges in $T$. We then consider the quadrature rule
\begin{equation}\label{eq:quadrature.rule.2D}
	\int_T v(\bmx) \dx \approx \sum_{F\in \Fh[T]} \sum_{i=1}^N \sum_{j=1}^M w_i^F(\bmx_i^F - \bmnu)\cdot \norT(\bmx_i^F) w_j x_j v(x_j \bmx_i^F + (1 - x_j)\bmnu).
\end{equation}
That is, we store weights
\[
w_i^F(\bmx_i^F - \bmnu)\cdot \norT(\bmx_i^F) w_j x_j,
\]
and abscissae
\[
x_j \bmx_i^F + (1 - x_j)\bmnu,
\]
for each $i=1,\dots,N$, $j=1,\dots,M$ and on each edge $F\in\Fh[T]$ that is not a straight edge connected to the vertex $\bmnu$. 

If $T$ is polygonal, then there always exists two straight edges connected to a vertex $\bmnu$. Thus, the rule described by \eqref{eq:quadrature.rule.2D} consists of $(|\Fh[T]| - 2) NM$ quadrature points. If we consider a Gauss-Legendre rule on each edge which is exact for polynomials of degree $k$, then we require to take $N = \lceil \frac{k + 1}{2} \rceil$. However, for the inverse divergence formula \eqref{eq:div.inv.rule.general} to reproduce polynomials of degree $k$ exactly, we require to take $M = \lceil \frac{k + 2}{2} \rceil$ due to the presence of the multiplier $t$. As $(\bmx_i^F - \bmnu)\cdot \norT(\bmx_i^F)$ is constant on polygonal $T$, equation \eqref{eq:quadrature.rule.2D} is exact for polynomials of degree $k$ and consists of $(|\Fh[T]| - 2) \lceil \frac{k + 1}{2} \rceil \lceil \frac{k + 2}{2} \rceil$ quadrature points.  We note this is a slightly larger number of quadrature points than the usual $(|\Fh[T]| - 2) \lceil \frac{k + 1}{2} \rceil^2$ required by splitting the polygon $T$ into $(|\Fh[T]| - 2)$ sub-triangles (an optimal sub-triangulation) and considering a Gauss-Legendre rule on each sub-triangle. However, \eqref{eq:quadrature.rule.2D} avoids the complex process of generating such a sub-triangulation. To avoid these additional quadrature points, one would need to consider a Gauss-Legendre rule on each edge, but a weighted Gaussian rule with the weight function $w(t)=t$ for the integral \eqref{eq:div.inv.rule.general}. This is not explored further here.



\subsection{A quadrature rule for elements in three dimensions}

In three dimensions, a volumetric integral can be mapped onto the faces as follows,
\begin{equation}\label{eq:boundary.integral.rule.3D}
	\int_{T} v(\bmx) \dx = \sum_{F\in\Fh[T]}\int_F (\bmx - \bmx_0) \cdot \norTF \int_{0}^1 t^{2} v(t\bmx + (1-t)\bmx_0) \dt \dF.
\end{equation}
Thus, given a quadrature rule for each face $F\in\Fh[T]$, a quadrature rule for the element $T$ can be developed analogously to the two-dimensional case. On each face $F\in\Fh[T]$ let us define $v_F(\bmx) = (\bmx - \bmx_0) \cdot \norTF \int_{0}^1 t^{2} v(t\bmx + (1-t)\bmx_0) \dt$. 
Take the planar region $\hat{F}\subset \bbR^2$ with (potentially curved) edges $\hat{\calE}_{\hat{F}}$ and a parameterisation $\bmgamma_F : \hat{F} \to F$. It holds that
\[
	\int_{T} v(\bmx) \dx = \sum_{F\in\Fh[T]} \int_{F} v_F(\bmx) \dF = \sum_{\hat{F}\in\hat{\calF}_T} \int_{\hat{F}} v_F(\bmgamma_F(\hat{\bmx})) J(\hat{\bmx}) \dxhat,
\]
where $J(\hat{\bmx}) = \sqrt{\det \bmJ^t(\hat{\bmx}) \bmJ(\hat{\bmx})}$ 
and $\bmJ$ is the Jacobian matrix of the map $\bmgamma_F$.
It then follows from \eqref{eq:integration.rule.general} that 
\begin{multline}\label{eq:curved.face.integration.rule}
	\int_F v_F(\bmx) \dx = \int_{\hat{F}} v_F(\bmgamma_F(\hat{\bmx})) J(\hat{\bmx}) \dxhat \\ = \sum_{\hat{E}\in\hat{\calE}_{\hat{F}}} \int_{\hat{E}} (\hat{\bmx} - \hat{\bmx}_0) \cdot \nor_{\hat{F}\hat{E}} \int_0^1 s v_F(\bmgamma_F(s\hat{\bmx} + (1-s)\hat{\bmx}_0)) J(s\hat{\bmx} + (1-s)\hat{\bmx}_0) \ds \dEhat,
\end{multline}
where $\nor_{\hat{F}\hat{E}}$ denotes the unit normal directed out of $\hat{F}$ and towards $\hat{E}$. Therefore, given the parameterisation $\bmgamma_F$ and a parameterisation of each mapped edge $\hat{E}\in\hat{\calE}_{\hat{F}}$, the integral \eqref{eq:curved.face.integration.rule} can be evaluated analogously to the 2D case \eqref{eq:quadrature.rule.2D}.
%
%
\subsubsection{A note on planar faces}

If the face $F$ is planar, one can follow a procedure similar to that in \cite{antonietti.houston:2018:fast} to rewrite the integrals on each face onto the edges $E\in\calE_F$. We take $\bmgamma_F(\hat{\bmx}) = \bmx_F + \bmE \hat{\bmx}$ where $\bmx_F$ is a point in the face $F$ and $\bmE$ is an orthonormal matrix. Then it holds that $J(\hat{\bmx})\equiv1$ and 
\[
	\bmgamma_F(s\hat{\bmx} + (1-s)\hat{\bmx}_0) = s\bmgamma_F(\hat{\bmx}) + (1 - s)\bmgamma_F(\hat{\bmx}_0).
\]
Thus, we can map the integral \eqref{eq:curved.face.integration.rule} back to the edges of the face $F$ as follows,
\[
	\int_F v_F(\bmx) = \sum_{E\in \calE_F} \int_E (\bmgamma_F^{-1}(\bmx) - \hat{\bmx}_0) \cdot \nor_{\hat{F}\hat{E}} \int_0^1 s v_F(s\bmx + (1-s)\bmgamma_F(\hat{\bmx}_0)) \ds\dE.
\]
However, as $\bmE$ is orthonormal it preserves distance and therefore it holds that $(\bmgamma_F^{-1}(\bmx) - \hat{\bmx}_0) \cdot \nor_{\hat{F}\hat{E}} = (\bmx - \bmgamma_F(\hat{\bmx}_0)) \cdot \nor_{FE}$, where $\nor_{FE}$ denotes the unit normal directed out of $F$ and towards $E$. Moreover, the mapping $\bmgamma_F$ is onto, so we can choose $\hat{\bmx}_0$ such that $\bmgamma_F(\hat{\bmx}_0) = \bmx_{F,0}$ for an arbitrary point $\bmx_{F,0}\in F$. Therefore
\begin{equation}\label{eq:planar.face.integration.rule}
	\int_F v_F(\bmx) = \sum_{E\in \calE_F} \int_E (\bmx - \bmx_{F,0}) \cdot \nor_{FE} \int_0^1 s v_F(s\bmx + (1-s)\bmx_{F,0}) \ds\dE.
\end{equation}
Again, we may choose $\bmx_{F,0}$ to be the vertex of the face $F$ connected to the largest number of straight edges. The integral \eqref{eq:planar.face.integration.rule} is then evaluated in an identical manner as two-dimensional elements.

\section{Implementation}\label{sec:implementation}

The HHO method for curved edges is implemented using the open source C++ library PolyMesh \cite{PolyMesh}. We generate curved meshes by first considering uniform Cartesian meshes and `cutting' along a curve. The integrals are computed using the quadrature rule described by \eqref{eq:quadrature.rule.2D} where we take the one-dimensional integration rules to be Gauss-Legendre rules of degree $30$.

A basis is formed for the space $\calP^{k}(F)$ by first generating a spanning set by considering a canonical basis of $\POLY{k}(\Omega)^d$ and taking $\POLY{0}(F) + \POLY{k}(\Omega)^d\cdot \norF$. The linearly dependent basis functions are removed algebraically using the FullPivLU class found in the \texttt{Eigen} library, with documentation available at \url{https://eigen.tuxfamily.org/dox/classEigen_1_1FullPivLU.html}. This requires a threshold to be set which determines the point at which pivots are considered to be numerically zero. We set this value to $10^{-15}$. We note that for sufficiently small $h$ and large $k$ this can result in certain linearly \emph{in}dependent functions being removed from $\calP^{k}(F)$. However, as these functions are `close' to being linearly dependent, the method seems unaffected by their removal. The bases of both $\calP^{k}(F)$ and $\POLY{k}(T)$ are orthonormalised via a Gram-Schmidt process.

\subsection{Curved boundary}
We consider here the domain given by the rotated ellipse
\begin{equation}
	\Omega = \{(x,y)\in\bbR^2 : L(x,y) > 0\},
\end{equation}
where the level set $L:\R^2\to\R$ is defined by
\[
	L(x,y) = \alpha^2 - (x^2 + xy + y^2),
\]
with $\alpha = \frac{4}{5}$. We note the following parameterisation of $\partial\Omega$: $\gamma: [0, 2 \pi) \to \partial\Omega$,
\[
	\gamma(t) = \alpha\big(\frac{1}{\sqrt{3}} \cos(t) - \sin(t), \frac{1}{\sqrt{3}} \cos(t) + \sin(t)\big).
\]
The exact solution to problem \eqref{eq:weak.form} is taken to be
\[
	u = \sin\big(L(x,y)\big),
\]
with corresponding source term given by
\begin{align*}
	f \eq (-\Delta L) \cos\big(L(x,y)\big) + |\nabla L|^2 \sin\big(L(x,y)\big) \nl \eq 4 \cos\big(L(x,y)\big) + \big(5x^2+8xy+5y^2\big) \sin\big(L(x,y)\big).
\end{align*}

The relative error of the scheme is measured through the following three quantities:
\begin{equation*}
	E_{0,h} \defeq \frac{\norm[\LTWO(\Th)]{u - \pKh{k+1}\uluh}}{\norm[\LTWO(\Th)]{u}} \quad;\quad
	E_{1,h} \defeq \frac{\seminorm[\HONE(\Th)]{u - \pKh{k+1}\uluh}}{\seminorm[\HONE(\Th)]{u}}\quad;\quad
	E_{\a,h} \defeq \frac{\norm[\a,\matK,h]{\uluh - \Ihk u}}{\norm[\a,\matK,h]{\Ihk u}},
\end{equation*}
where the norm $\norm[\LTWO(\Th)]{\cdot}$ 
is defined as the square-root of the sum of squares of $\norm[\LTWO(T)]{\cdot}$.
We note that if the mesh conforms to the domain $\Omega$ then $\norm[\LTWO(\Th)]{v} = \norm[\Omega]{v}$ for all $v\in \LTWO(\Omega)$.

We consider here two sequence of meshes of the domain $\Omega$. The curved meshes use an exact representation of the boundary, whereas the straight meshes take a piece-wise linear approximation of the boundary. The parameters of the mesh sequences are displayed in Table \ref{table:mesh.curved}. Both sequences of meshes have the same parameters. Example curved meshes are plotted in Figure \ref{fig:mesh.curved} and straight meshes are plotted in Figure \ref{fig:mesh.straight}.

\begin{table}[H]
\centering
\pgfplotstableread{data/poisson_mesh_data.dat}\loadedtable
\pgfplotstabletypeset
[
columns={MeshTitle,CurvedMeshSize,NbCells,NbInternalEdges}, 
columns/MeshTitle/.style={column name=Mesh \#},
columns/CurvedMeshSize/.style={column name=\(h\),/pgf/number format/.cd,fixed,zerofill,precision=4},
columns/NbCells/.style={column name=Nb. Elements},
columns/NbInternalEdges/.style={column name=Nb. Internal Edges},
every head row/.style={before row=\toprule,after row=\midrule},
every last row/.style={after row=\bottomrule} 
]\loadedtable
\caption{Parameters of the mesh sequences used for the curved boundary test}
\label{table:mesh.curved}
\end{table}

\begin{figure}[H]
	\centering
	\includegraphics[width=0.3\textwidth]{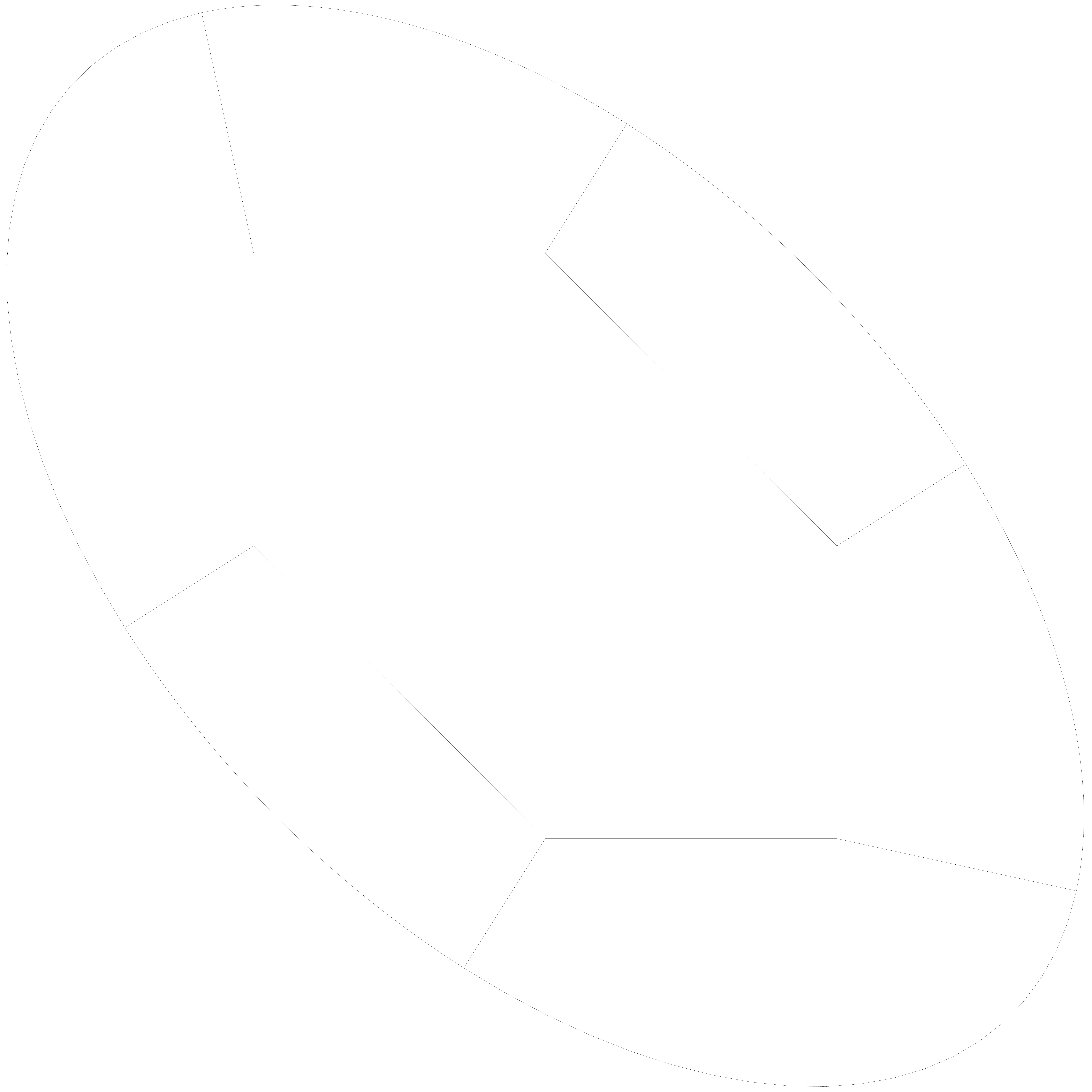} 
	\includegraphics[width=0.3\textwidth]{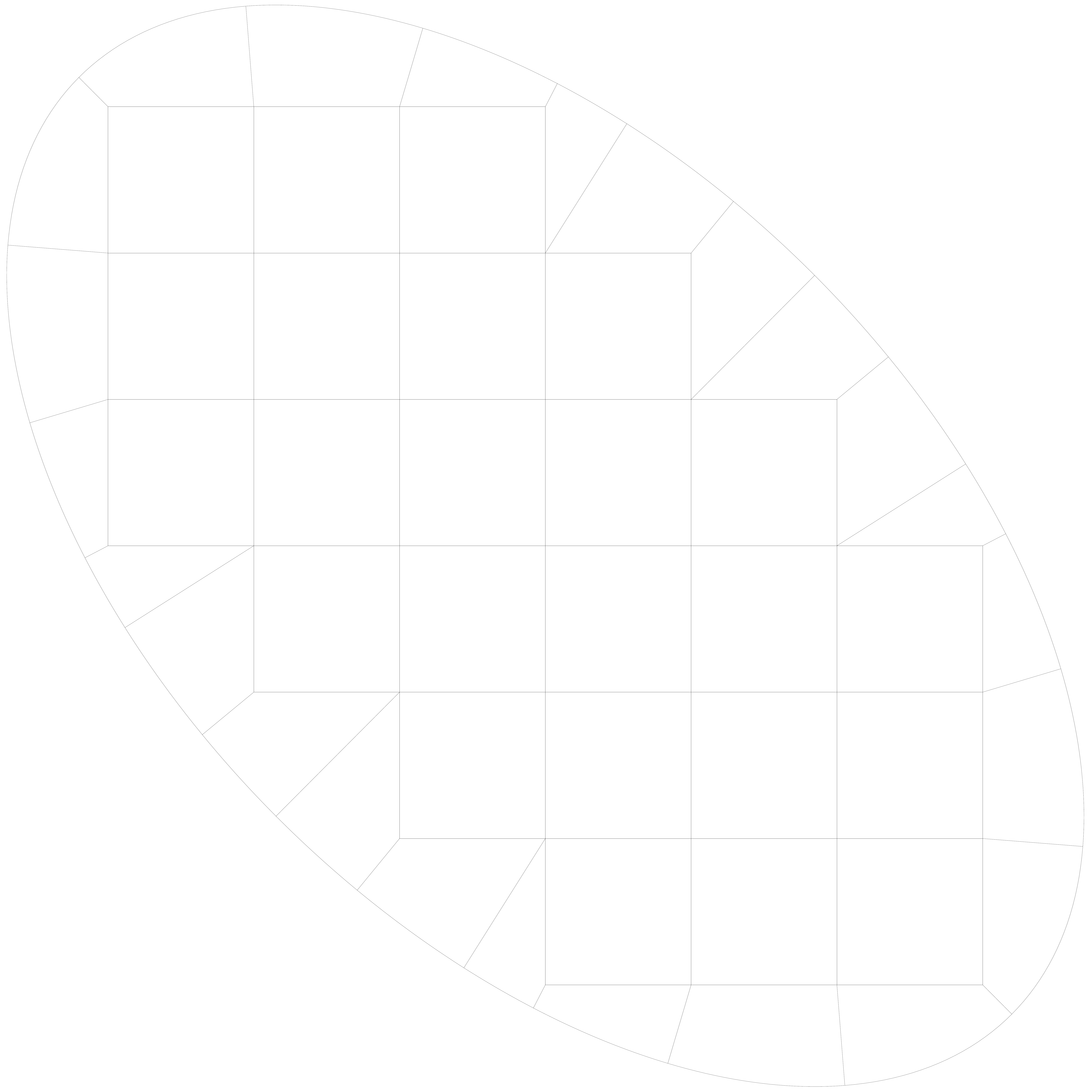} 
	\includegraphics[width=0.3\textwidth]{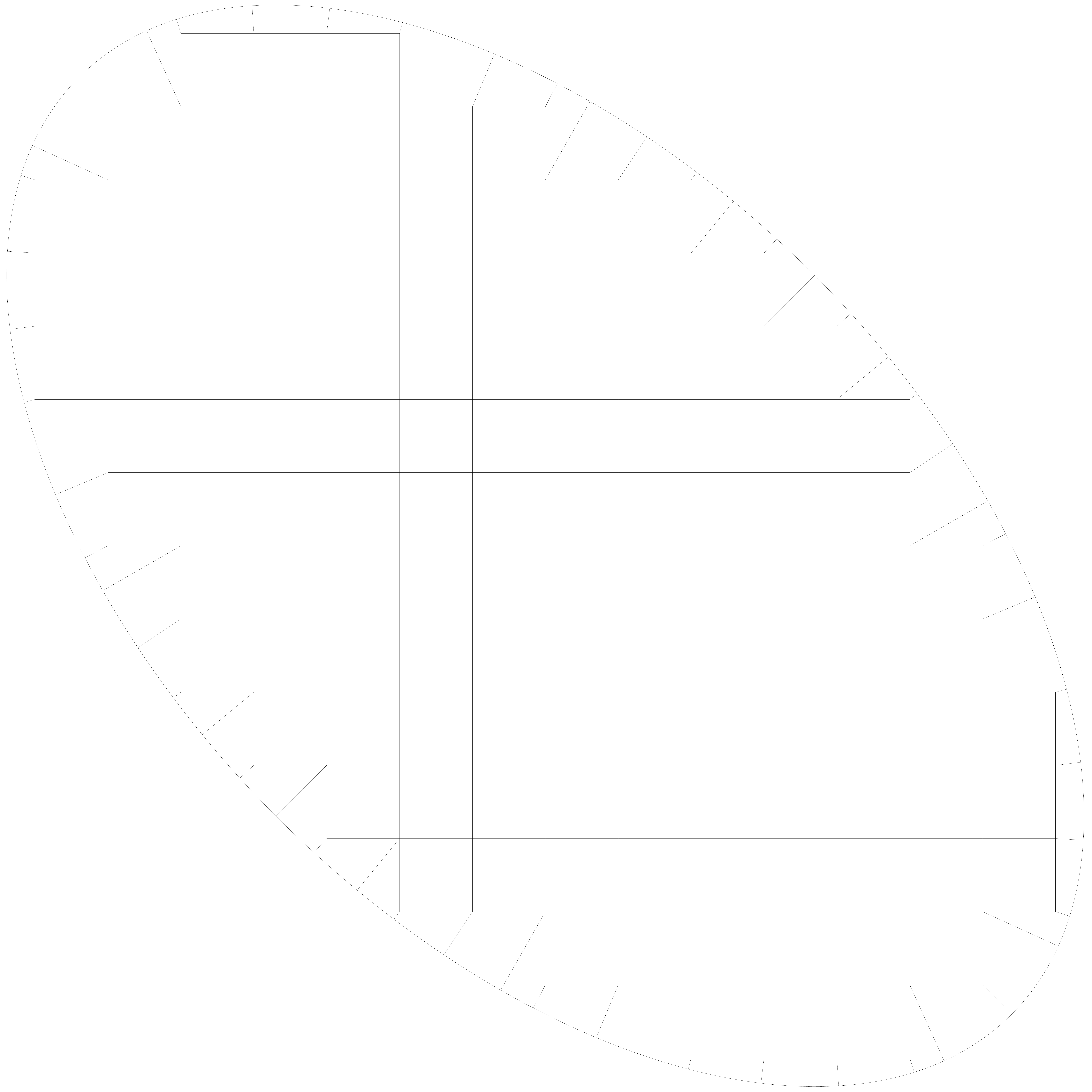} 
	\caption{Example curved meshes used for the curved boundary test}
	\label{fig:mesh.curved}
\end{figure}
\begin{figure}[H]
	\centering
	\includegraphics[width=0.3\textwidth]{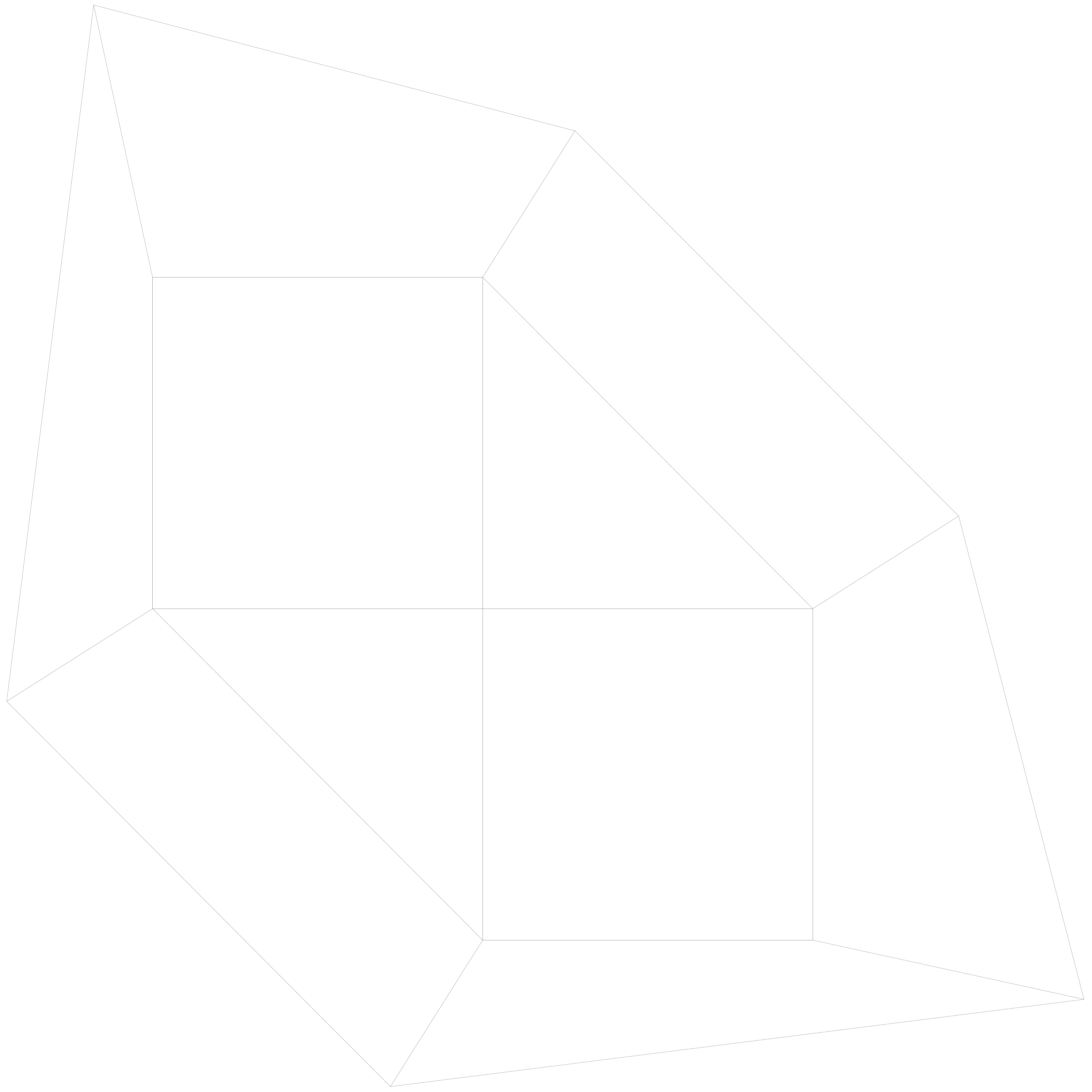} 
	\includegraphics[width=0.3\textwidth]{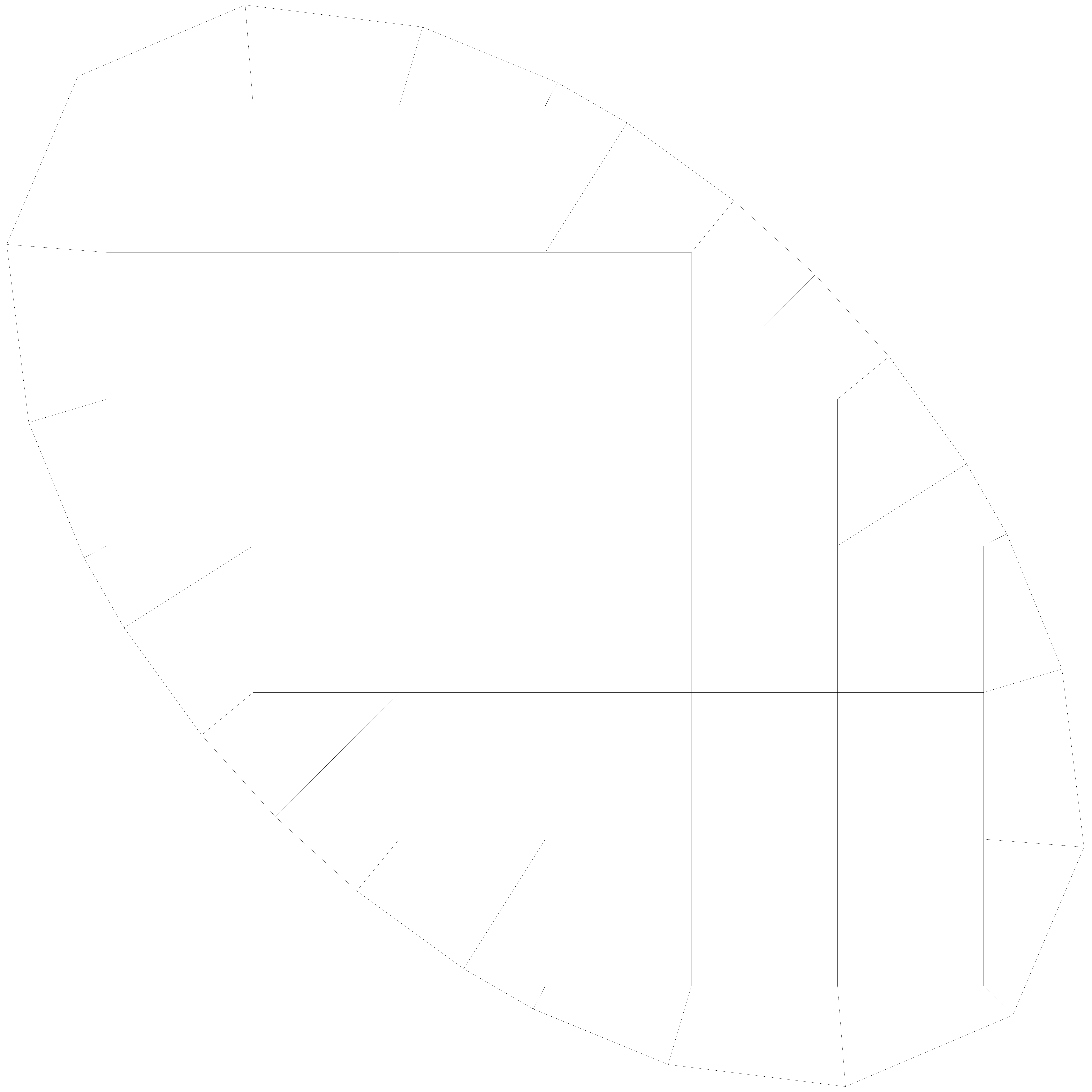} 
	\includegraphics[width=0.3\textwidth]{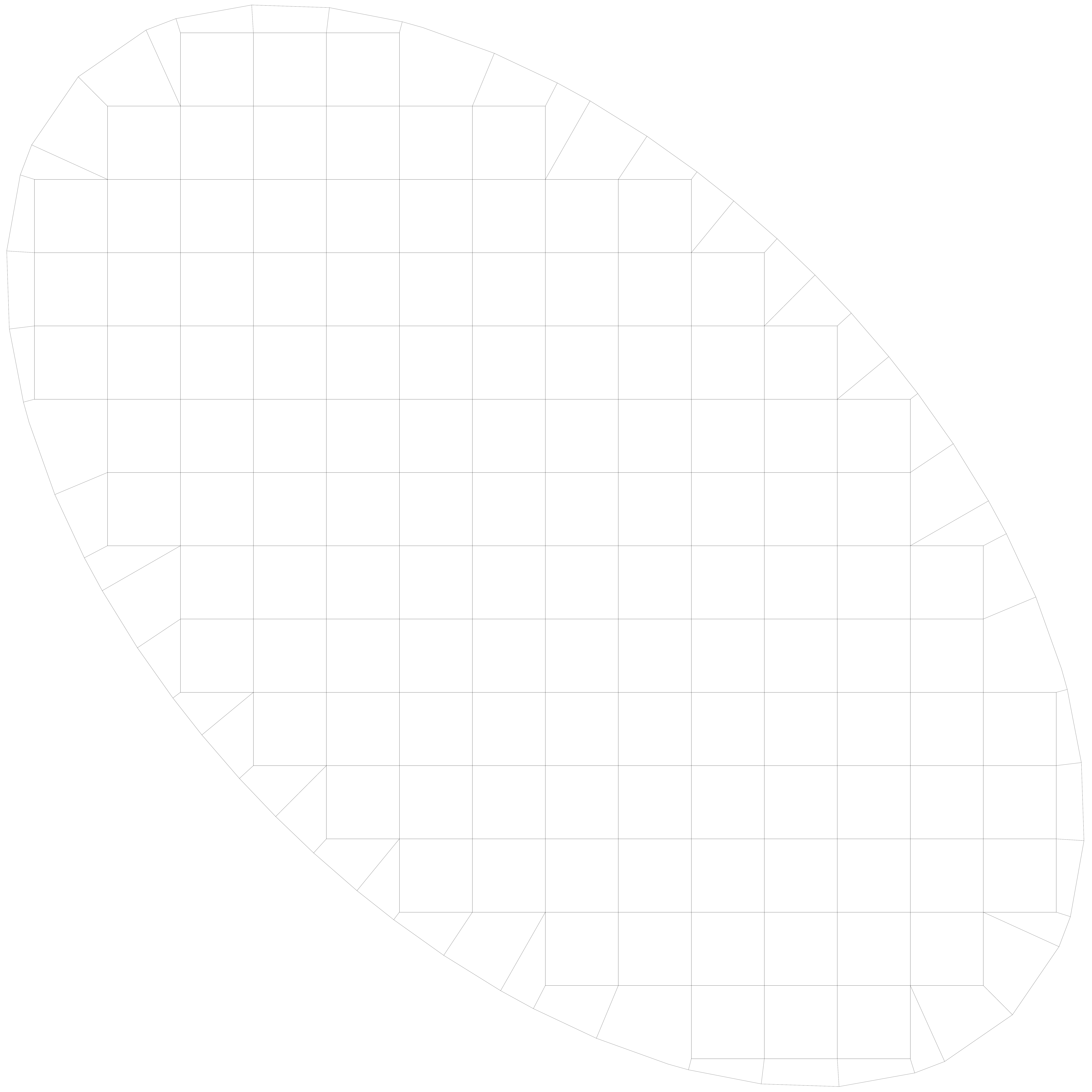} 
	\caption{Example straight meshes used for the curved boundary test}
	\label{fig:mesh.straight}
\end{figure}

In Figure \ref{fig:htest} we test both a curved HHO scheme and a classical HHO scheme (on straight meshes) with polynomial degrees given by $k=1$ and $k=3$. In both cases the curved HHO scheme on the fitted mesh observes significantly better convergence rates than the classical scheme on the straight mesh. While the scheme appears to converge optimally on curved meshes, it converges at most order $2$ on straight meshes. 

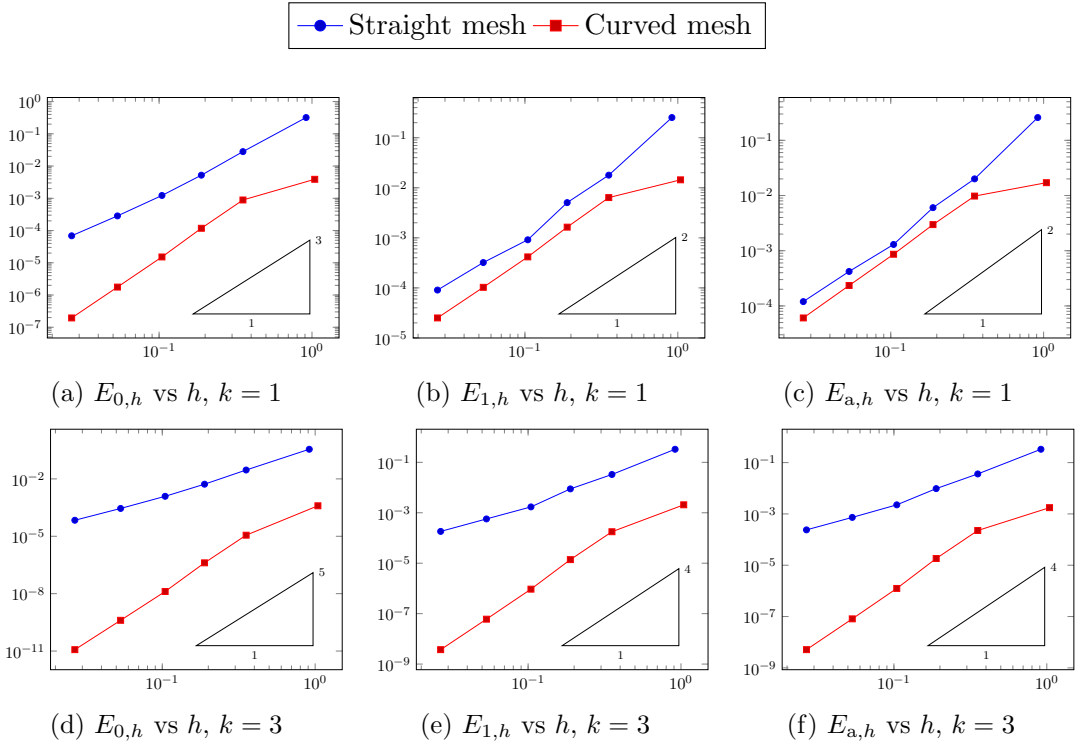
\begin{figure}[H]
	\centering
	\ref{graph.legend}
	\vspace{0.5cm}\\
	\subcaptionbox{$E_{0,h}$ vs $h$, \(k=1\)}
	{
		\begin{tikzpicture}[scale=0.56]
			\begin{loglogaxis}[ legend columns=2, legend to name=graph.legend ]
				\addplot table[x=MeshSize,y=L2Error] {data/straight_k1.dat};
				\addplot table[x=MeshSize,y=L2Error] {data/curved_k1.dat};
				\logLogSlopeTriangle{0.90}{0.4}{0.1}{3}{black};
				\legend{Straight mesh, Curved mesh};
			\end{loglogaxis}
		\end{tikzpicture}
	}
	\subcaptionbox{$E_{1,h}$ vs $h$, \(k=1\)}
	{
		\begin{tikzpicture}[scale=0.56]
			\begin{loglogaxis}
				\addplot table[x=MeshSize,y=H1Error] {data/straight_k1.dat};
				\addplot table[x=MeshSize,y=H1Error] {data/curved_k1.dat};
				\logLogSlopeTriangle{0.90}{0.4}{0.1}{2}{black};
			\end{loglogaxis}
		\end{tikzpicture}
	}
	\subcaptionbox{$E_{\a,h}$ vs $h$, \(k=1\)}
	{
		\begin{tikzpicture}[scale=0.56]
			\begin{loglogaxis}
				\addplot table[x=MeshSize,y=EnergyError] {data/straight_k1.dat};
				\addplot table[x=MeshSize,y=EnergyError] {data/curved_k1.dat};
				\logLogSlopeTriangle{0.90}{0.4}{0.1}{2}{black};
			\end{loglogaxis}
		\end{tikzpicture}
	}
\\
\medskip
	\subcaptionbox{$E_{0,h}$ vs $h$, \(k=3\)}
	{
		\begin{tikzpicture}[scale=0.56]
			\begin{loglogaxis}
				\addplot table[x=MeshSize,y=L2Error] {data/straight_k3.dat};
				\addplot table[x=MeshSize,y=L2Error] {data/curved_k3.dat};
				\logLogSlopeTriangle{0.90}{0.4}{0.1}{5}{black};
			\end{loglogaxis}
		\end{tikzpicture}
	}
	\subcaptionbox{$E_{1,h}$ vs $h$, \(k=3\)}
	{
		\begin{tikzpicture}[scale=0.56]
			\begin{loglogaxis}
				\addplot table[x=MeshSize,y=H1Error] {data/straight_k3.dat};
				\addplot table[x=MeshSize,y=H1Error] {data/curved_k3.dat};
				\logLogSlopeTriangle{0.90}{0.4}{0.1}{4}{black};
			\end{loglogaxis}
		\end{tikzpicture}
	}
	\subcaptionbox{$E_{\a,h}$ vs $h$, \(k=3\)}
	{
		\begin{tikzpicture}[scale=0.56]
			\begin{loglogaxis}
				\addplot table[x=MeshSize,y=EnergyError] {data/straight_k3.dat};
				\addplot table[x=MeshSize,y=EnergyError] {data/curved_k3.dat};
				\logLogSlopeTriangle{0.90}{0.4}{0.1}{4}{black};
			\end{loglogaxis}
		\end{tikzpicture}
	}
	\caption{$h$-version curved boundary test}
	\label{fig:htest}
\end{figure}

In Figure \ref{fig:ktest.mesh2} we test the performance of both methods as $k$ increases on Mesh 2. While the scheme enjoys exponential convergence on the curved mesh, the classical method on the straight mesh does not converge. This is to be expected, as the straight mesh does not fit $\Omega$ exactly and so by increasing $k$ the scheme is converging to the solution of a different problem.


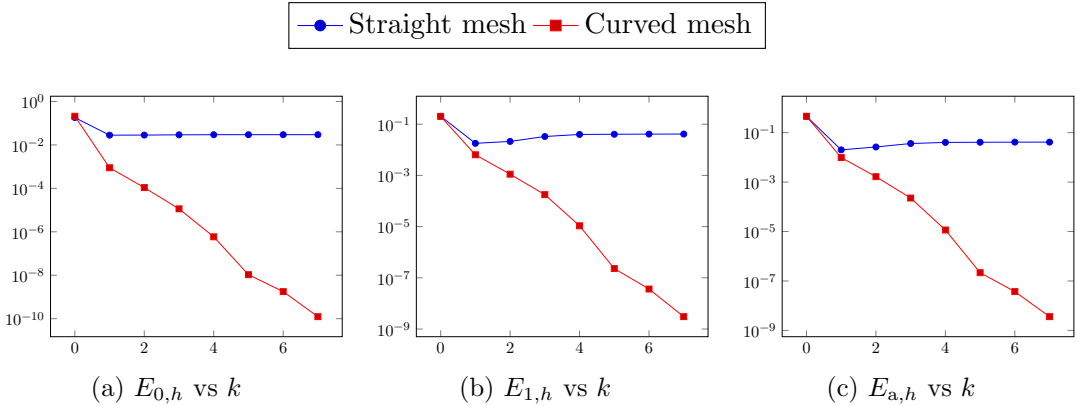
\begin{figure}[H]
	\centering
	\ref{graph.legend}
	\vspace{0.5cm}\\
	\subcaptionbox{$E_{0,h}$ vs $k$}
	{
		\begin{tikzpicture}[scale=0.56]
			\begin{semilogyaxis}
				\addplot table[x=EdgeDegree,y=L2Error] {data/straight_ktest_mesh2_2.dat};
				\addplot table[x=EdgeDegree,y=L2Error] {data/curved_ktest_mesh2_2.dat};
			\end{semilogyaxis}
		\end{tikzpicture}
	}
	\subcaptionbox{$E_{1,h}$ vs $k$}
	{
		\begin{tikzpicture}[scale=0.56]
			\begin{semilogyaxis}
				\addplot table[x=EdgeDegree,y=H1Error] {data/straight_ktest_mesh2_2.dat};
				\addplot table[x=EdgeDegree,y=H1Error] {data/curved_ktest_mesh2_2.dat};
			\end{semilogyaxis}
		\end{tikzpicture}
	}
	\subcaptionbox{$E_{\a,h}$ vs $k$}
	{
		\begin{tikzpicture}[scale=0.56]
			\begin{semilogyaxis}
				\addplot table[x=EdgeDegree,y=EnergyError] {data/straight_ktest_mesh2_2.dat};
				\addplot table[x=EdgeDegree,y=EnergyError] {data/curved_ktest_mesh2_2.dat};
			\end{semilogyaxis}
		\end{tikzpicture}
	}
	\caption{$k$-version curved boundary test on Mesh 2 (\(h\approx 0.3536\))}
	\label{fig:ktest.mesh2}
\end{figure}

\subsection{Heterogeneous diffusion}

We conclude the numerical section with a test of a diffusion problem with a piece-wise constant diffusion tensor. The HHO method requires the mesh to conform to any discontinuities in the diffusion. Thus, if the diffusion has a discontinuity along a curve, the mesh has to be curved to fit the discontinuity in the diffusion. Any polytopal mesh will require an approximation of the diffusion tensor. 

We consider $\Omega=\{(x,y):x^2+y^2<1\}$ to be the unit disc and $\matK$ a piece-wise constant diffusion tensor given by
\[
\matK = 
\begin{cases}
	\begin{pmatrix}
			1 & 1-\beta_1 \\ 1-\beta_1 & 1
		\end{pmatrix}
	\quad \textrm{if } r < R \\
	\medskip\\
	\begin{pmatrix}
			1 & 1-\beta_2 \\ 1-\beta_2 & 1
		\end{pmatrix}
	\quad \textrm{if } r > R 
\end{cases}.
\]
We take $R = 0.8$, $\beta_1=10^{-6}$ and $\beta_2=1$ which corresponds to anisotropic diffusion in the region $r < R$, and a Poisson problem in $r > R$. We take the source term to be $f \equiv 1$.

Again, we consider two sequences of meshes of the domain $\Omega$. We take both sequences to fit the domain $\Omega$ exactly, however, the curved mesh we take to fit the discontinuity in $\bmK$ exactly and the straight mesh takes a piece-wise linear approximation of $\bmK$. The mesh data is presented in Table \ref{table:mesh.curved.diffusion}. We note that both sequences of meshes have the same parameters.
\begin{table}[H]
	\centering
	\pgfplotstableread{data/diffusion_mesh_data.dat}\loadedtable
	\pgfplotstabletypeset
	[
	columns={MeshTitle,MeshSize,NbCells,NbInternalEdges}, 
	columns/MeshTitle/.style={column name=Mesh \#},
	columns/MeshSize/.style={column name=\(h\),/pgf/number format/.cd,fixed,zerofill,precision=4},
	columns/NbCells/.style={column name=Nb. Elements},
	columns/NbInternalEdges/.style={column name=Nb. Internal Edges},
	every head row/.style={before row=\toprule,after row=\midrule},
	every last row/.style={after row=\bottomrule} 
	]\loadedtable
	\caption{Parameters of the mesh sequences used for the heterogeneous diffusion test}
	\label{table:mesh.curved.diffusion}
\end{table}
An example curved mesh and an example straight mesh is plotted in Figure \ref{fig:mesh.diffusion}.
\begin{figure}[H]
	\centering
	\includegraphics[width=0.35\textwidth]{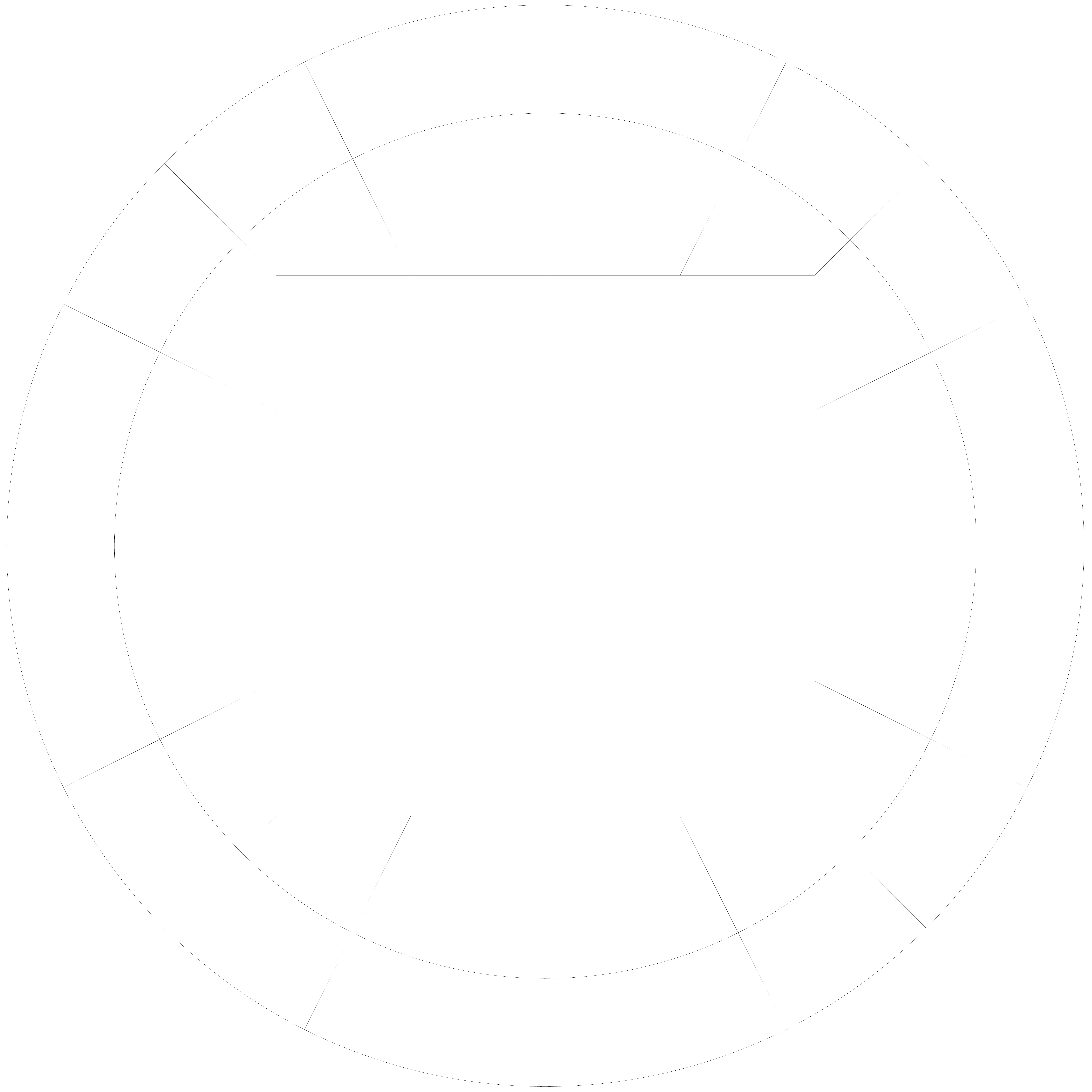} \hspace{1cm}
	\includegraphics[width=0.35\textwidth]{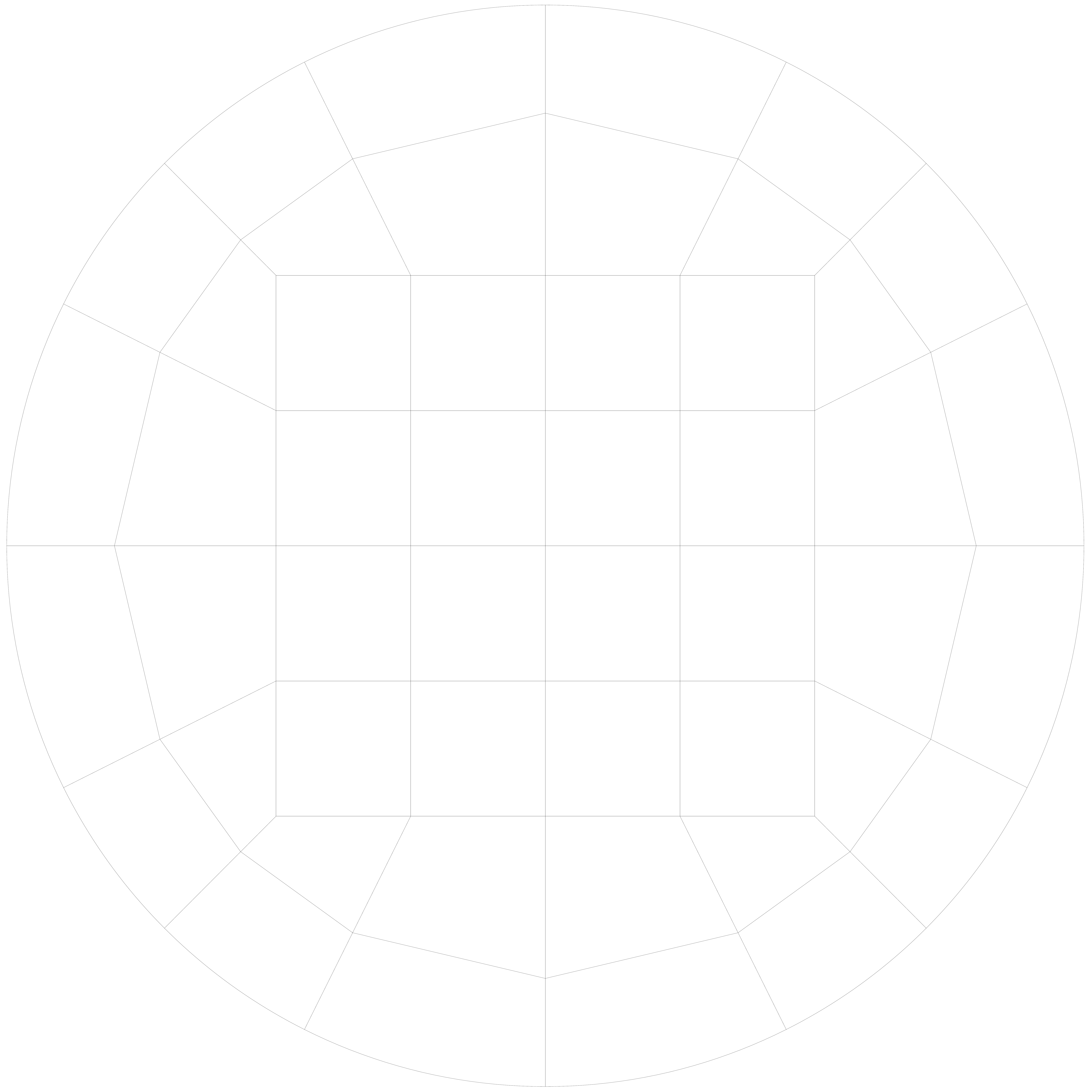} 
	\caption{Example meshes used for the heterogeneous diffusion test}
	\label{fig:mesh.diffusion}
\end{figure}

As we do not know the exact solution to this problem, we run the scheme on the finest curved mesh with $k=7$. We denote by the discrete solution to this problem $\pKh{k+1}\uluh = u_h^*$, which will play the role of the `exact' solution. We measure the quantities
\[
	\int_{\Omega} u_h^* \approx 0.46006947 \quad ; \quad \seminorm[\HONE(\Th)]{u_h^*} \approx 0.80699766.
\]
We would then like to test the performance of the scheme on coarser meshes (both curved and straight) with smaller $k$ by investigating the behaviour of
\[
	E_1 = \SEMINORM{\int_{\Omega} (\pKh{k+1}\uluh - u_h^*)} \quad\textrm{and}\quad E_2 = \SEMINORM{\seminorm[\HONE(\Th)]{\pKh{k+1}\uluh} - \seminorm[\HONE(\Th)]{u_h^*}}.
\]
We are less interested in the rate of convergence of these measures, but rather want to observe steady convergence, and investigate the difference between the two schemes. In Figure \ref{fig:ktest.diffusion.mesh1} we plot the quantities $E_1$ and $E_2$ against increasing polynomial degree $k$ where we fix the mesh to be Mesh 1. It is clear that the scheme on the straight mesh, where we consider an approximate diffusion tensor, stops converging for $k > 1$ whereas the scheme on the curved mesh converges smoothly. In Figure \ref{fig:htest.diffusion} we test convergence against decreasing mesh size $h$ for polynomial degrees $k$. While for $k=1$ the order of $E_1$ and $E_2$ are similar for both schemes (and at times smaller on the straight mesh), the convergence is much smoother on the curved mesh. For $k=3$, the values are significantly smaller for the curved mesh, and the behaviour for the straight mesh does not differ much from the $k=1$ case. This coincides with the previous observations that increasing $k$ past $1$ has little effect on the scheme when considering a piece-wise linear approximation of the discontinuity in the diffusion.


\begin{figure}[H]
	\centering
	\ref{graph.legend}
	\vspace{0.5cm}\\
	\subcaptionbox{$E_1$ vs $k$}
	{
		\begin{tikzpicture}[scale=0.85]
			\begin{semilogyaxis}
				\addplot table[x=EdgeDegree,y=integral_error] {data/straight_ktest_source_1.dat};
				\addplot table[x=EdgeDegree,y=integral_error] {data/curved_ktest_source_1.dat};
			\end{semilogyaxis}
		\end{tikzpicture}
	}
	\hspace{0.1cm}
	\subcaptionbox{$E_2$ vs $k$}
	{
	\begin{tikzpicture}[scale=0.85]
		\begin{semilogyaxis}
			\addplot table[x=EdgeDegree,y=h1_norm_error] {data/straight_ktest_source_1.dat};
			\addplot table[x=EdgeDegree,y=h1_norm_error] {data/curved_ktest_source_1.dat};
		\end{semilogyaxis}
	\end{tikzpicture}
	}
	\caption{$k$-version heterogeneous diffusion test on Mesh 1 ($h\approx0.7654$)}
	\label{fig:ktest.diffusion.mesh1}
\end{figure}
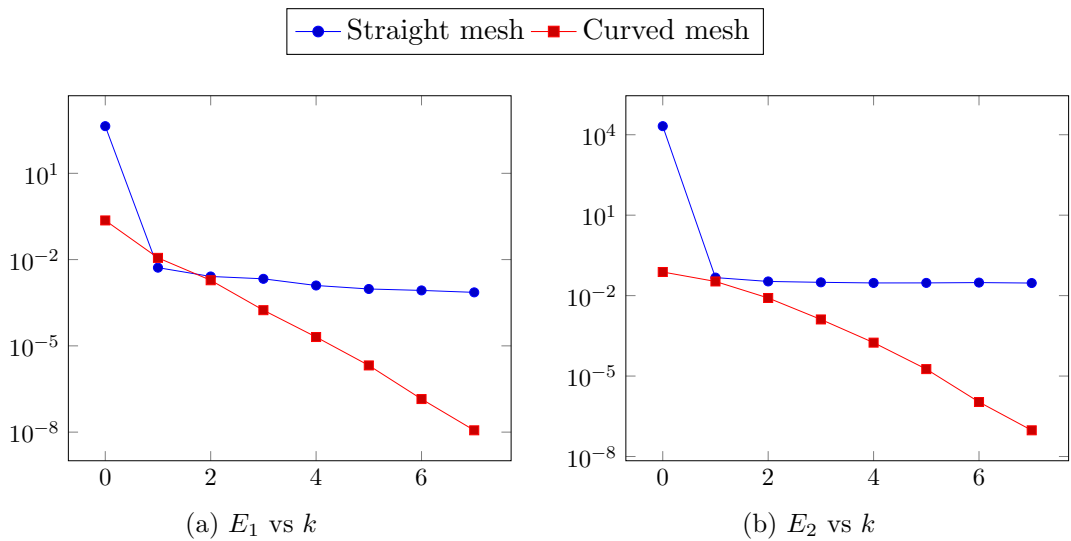

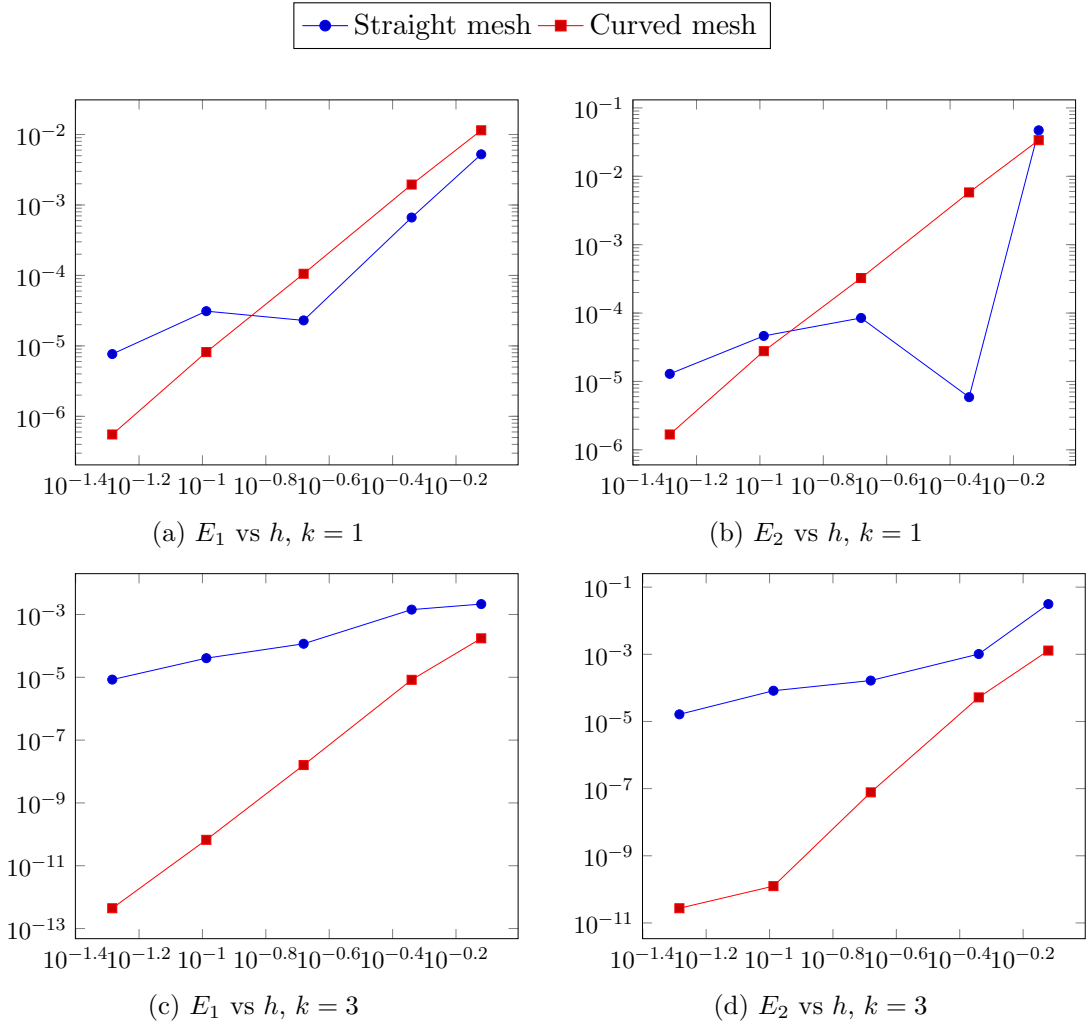
\begin{figure}[H]
\centering
\ref{graph.legend}
\vspace{0.5cm}\\
\subcaptionbox{$E_1$ vs $h$, $k=1$}
{
	\begin{tikzpicture}[scale=0.85]
		\begin{loglogaxis}
			\addplot table[x=MeshSize,y=integral_error] {data/straight_k1_source_1.dat};
			\addplot table[x=MeshSize,y=integral_error] {data/curved_k1_source_1.dat};
		\end{loglogaxis}
	\end{tikzpicture}
}
\hspace{0.1cm}
\subcaptionbox{\(E_2\) vs $h$, $k=1$}
{
	\begin{tikzpicture}[scale=0.85]
		\begin{loglogaxis}
			\addplot table[x=MeshSize,y=h1_norm_error] {data/straight_k1_source_1.dat};
			\addplot table[x=MeshSize,y=h1_norm_error] {data/curved_k1_source_1.dat};
		\end{loglogaxis}
	\end{tikzpicture}
}
\\
\medskip
%
\subcaptionbox{$E_1$ vs $h$, $k=3$}
{
	\begin{tikzpicture}[scale=0.85]
		\begin{loglogaxis}
			\addplot table[x=MeshSize,y=integral_error] {data/straight_k3_source_1.dat};
			\addplot table[x=MeshSize,y=integral_error] {data/curved_k3_source_1.dat};
		\end{loglogaxis}
	\end{tikzpicture}
}
\hspace{0.1cm}
\subcaptionbox{$E_2$ vs $h$, $k=3$}
{
	\begin{tikzpicture}[scale=0.85]
		\begin{loglogaxis}
			\addplot table[x=MeshSize,y=h1_norm_error] {data/straight_k3_source_1.dat};
			\addplot table[x=MeshSize,y=h1_norm_error] {data/curved_k3_source_1.dat};
		\end{loglogaxis}
	\end{tikzpicture}
}
\caption{$h$-version heterogeneous diffusion test}
\label{fig:htest.diffusion}
\end{figure}

Finally, in Figure \ref{fig:sol.plots.mesh1.k7} we show contour plots of the potential reconstructions of the discrete solutions on Mesh 1 with $k=7$. We observe that the plot on the straight mesh seems to be distorted along the eigen vectors of $\matK$ (that is, $(1,1)^t$ and $(1,-1)^t$) when compared to the plot on the curved mesh. We also plot the absolute value of the difference between the two schemes and observe that this value seems to be of greatest magnitude around the discontinuity in the diffusion tensor.

\begin{figure}[H]
\centering
\includegraphics[width=0.6\textwidth]{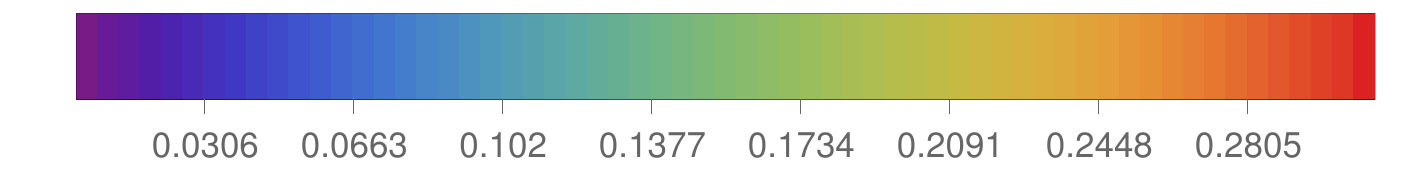}\\
\subcaptionbox{Curved mesh}
{
	\includegraphics[width=0.45\textwidth]{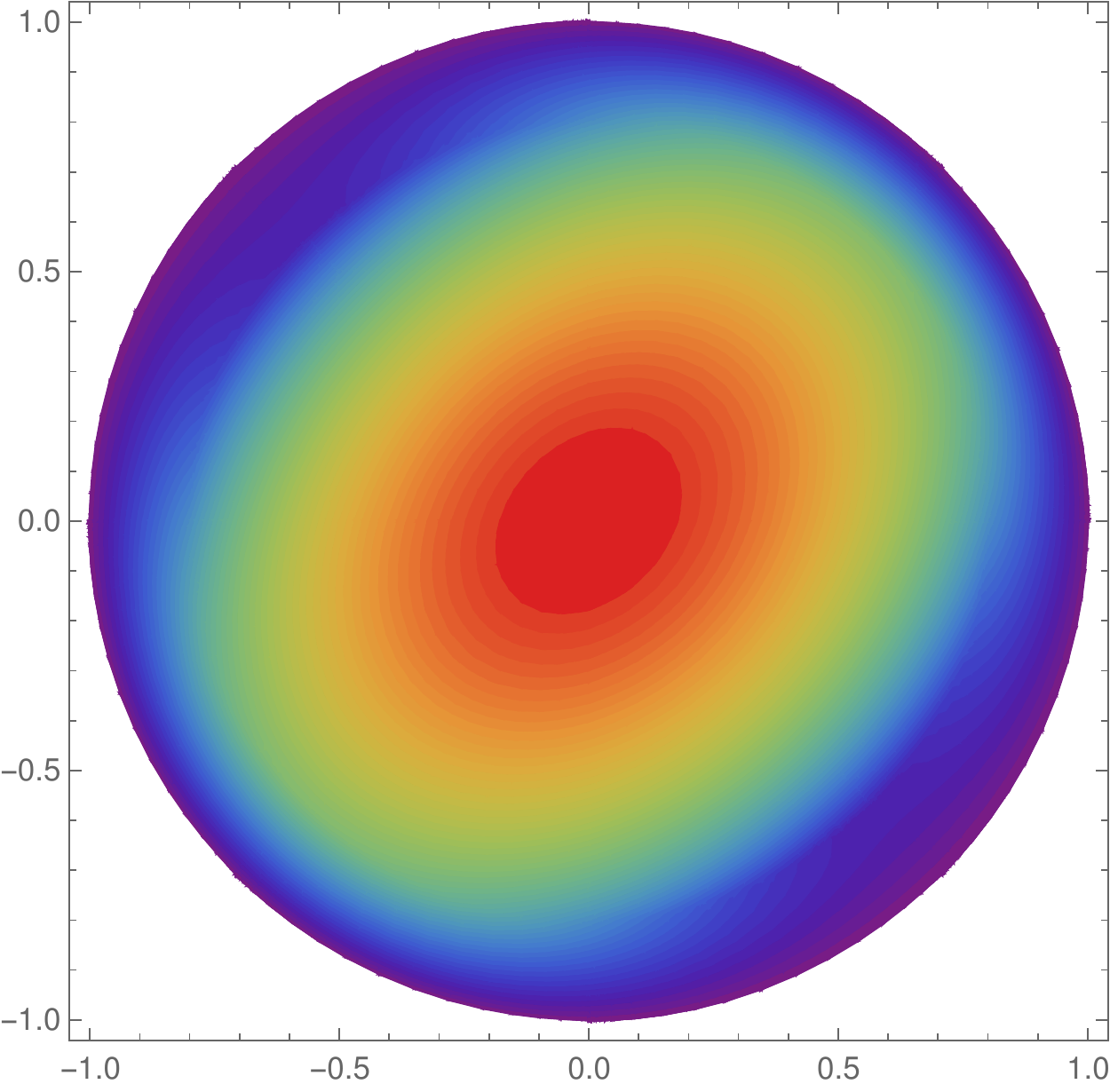}
}
\subcaptionbox{Straight mesh}
{
	\includegraphics[width=0.45\textwidth]{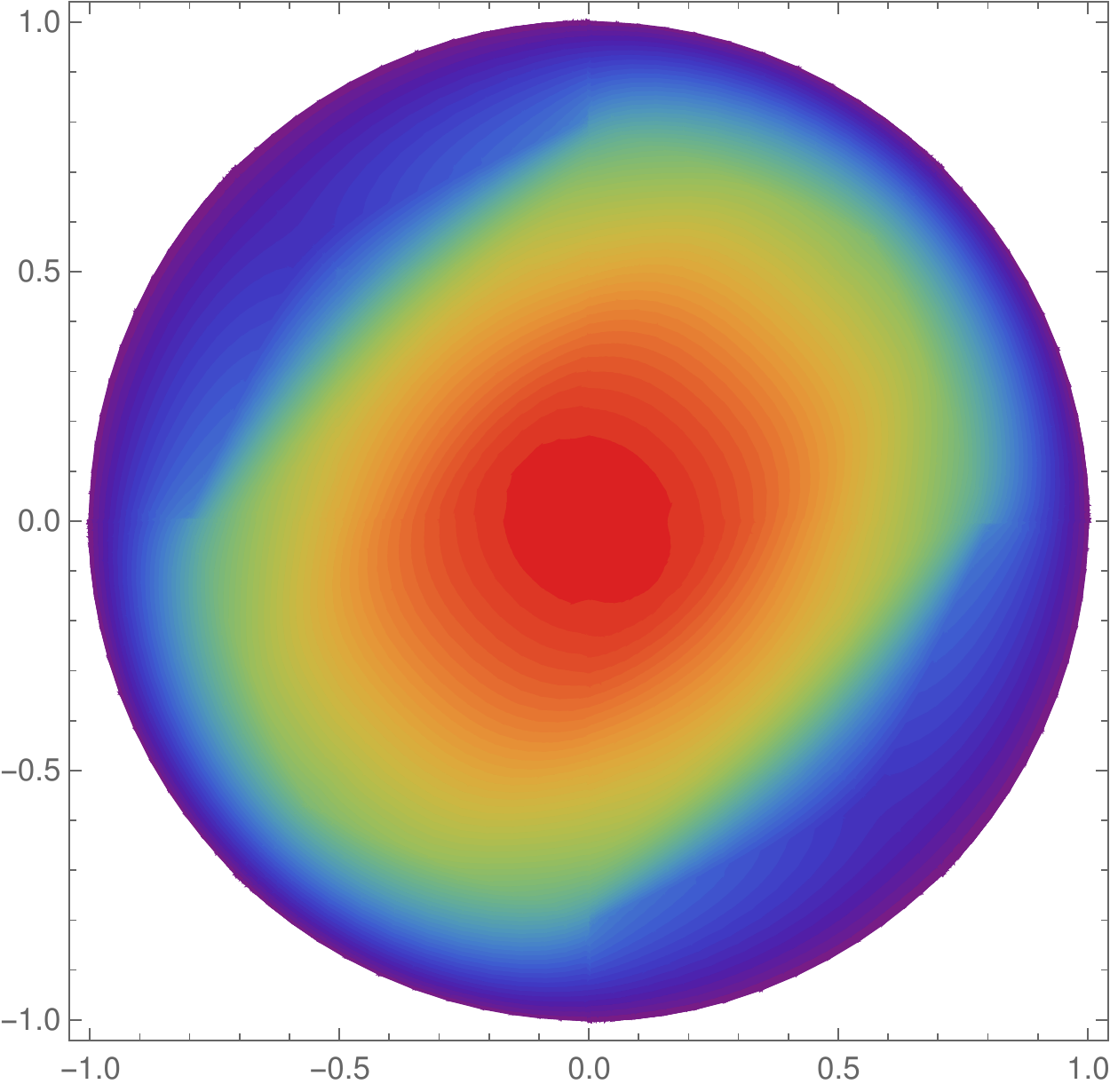}
}
\\
\medskip
\includegraphics[width=0.6\textwidth]{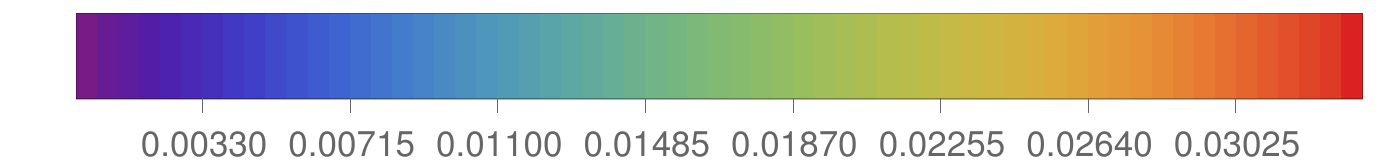}\\
\subcaptionbox{Difference}
{
	\includegraphics[width=0.45\textwidth]{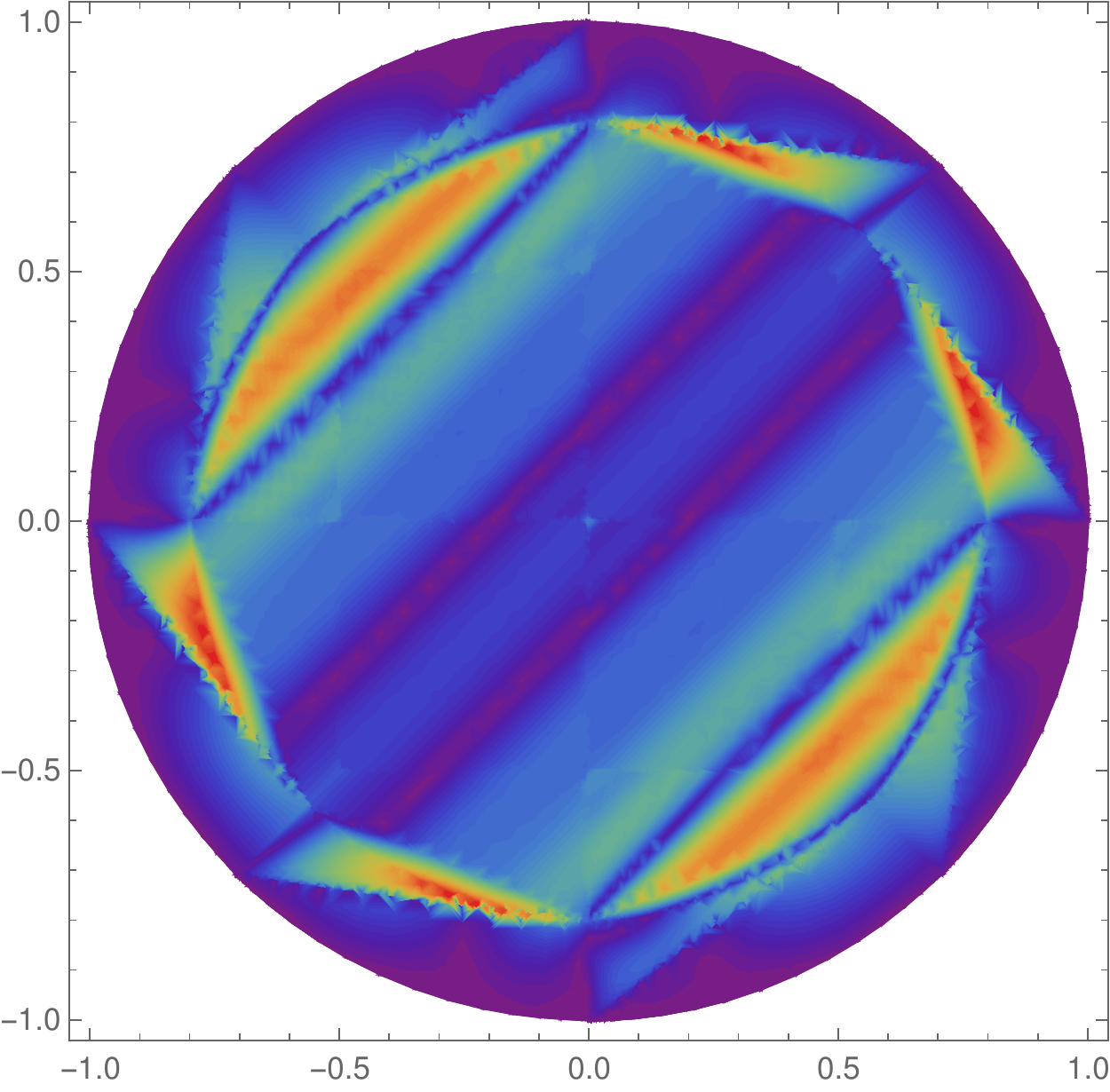}
}
\caption{Contour plots of the heterogeneous diffusion test on Mesh 1 with $k=7$}
\label{fig:sol.plots.mesh1.k7}
\end{figure}

\section*{Declarations}

The author declares that they have no conflict of interest.

\printbibliography

\end{document}

%% file: commands.tex
\usepackage[hmargin={25mm,25mm},vmargin={30mm,35mm}]{geometry}
\usepackage{amsfonts}
\usepackage{amssymb}
\usepackage{amsthm}
\usepackage{bm}
\usepackage{mathtools}
\usepackage{latexsym}
\usepackage{graphicx}
\usepackage{booktabs}
\usepackage{pgfplots,pgfplotstable}
\usepackage{xcolor}
\usepackage{setspace}
\usepackage{url}
\usepackage{float}
\usepackage{subcaption}
\usepackage{enumitem}
\usepackage{eufrak}

\usepackage[colorlinks,allcolors={blue},linkcolor={black}]{hyperref}

\usepackage[natbib=true,style=numeric-comp,backend=bibtex,maxnames=99,giveninits]{biblatex}

\numberwithin{equation}{section}
\usepackage[color,notref,notcite]{showkeys}

\usepackage[normalem]{ulem}
\newcounter{corr}
\definecolor{violet}{rgb}{0.580,0.,0.827}
\newcommand{\corr}[4][]{\typeout{Warning : a correction remains in page \thepage}
 \stepcounter{corr}
	      {\color{blue}\ifmmode\text{\,\sout{\ensuremath{#2}}\,}\else\sout{#2}\fi}
        {\color{green!50!black}#3}
        {\color{violet}#4}
}

\usepackage{authblk}
\newcommand{\email}[1]{\href{mailto:#1}{#1}}

\newtheorem{theorem}{Theorem}

\newtheorem{lemma}[theorem]{Lemma}

\newtheorem{assumption}{Assumption}

\newtheorem{remark}{Remark}

\newcommand{\bmrm}[1]{{\bm{{\rm #1}}}}

\newcommand{\mat}[1]{\bmrm{#1}}

\newcommand{\bmn}{{\bm{n}}}

\newcommand{\bmr}{{\bm{r}}}

\newcommand{\bmx}{{\bm{x}}}

\newcommand{\bmE}{\bm{E}}
\newcommand{\bmF}{\bm{F}}

\newcommand{\bmH}{\bm{H}}

\newcommand{\bmJ}{\bm{J}}
\newcommand{\bmK}{\bm{K}}

\newcommand{\bmgamma}{{\bm{\gamma}}}

\newcommand{\bmnu}{{\bm{\nu}}}

\newcommand{\calE}{\mathcal{E}}
\newcommand{\calF}{\mathcal{F}}

\newcommand{\calH}{\mathcal{H}}

\newcommand{\calM}{\mathcal{M}}

\newcommand{\calP}{\mathcal{P}}

\newcommand{\calT}{\mathcal{T}}

\newcommand{\bbN}{\mathbb{N}}

\newcommand{\bbP}{\mathbb{P}}

\newcommand{\bbR}{\mathbb{R}}

\newcommand{\rma}{{\rm{a}}}
\newcommand{\rmb}{{\rm{b}}}

\newcommand{\rmd}{{\rm{d}}}

\newcommand{\rmi}{{\rm{i}}}

\newcommand{\rmp}{{\rm{p}}}

\newcommand{\rms}{{\rm{s}}}

\newcommand{\dEhat}{\,{\rmd} \hat{E}}
\newcommand{\dE}{\,{\rmd} E}
\newcommand{\dF}{\,{\rmd} F}
\newcommand{\dS}{\,{\rmd} S}
\newcommand{\ds}{\,{\rmd} s}
\newcommand{\dt}{\,{\rmd} t}
\newcommand{\dx}{\,{\rmd} \bmx}
\newcommand{\dxhat}{\,{\rmd} \hat{\bmx}}

\newcommand{\eq}{ ={}& }
\newcommand{\lea}{ \le{}& }

\newcommand{\les}{ \lesssim{}& }
\newcommand{\plus}{ &{}+ }

\newcommand{\nn}{\nonumber}
\newcommand{\nl}{\nn\\}

\newcommand{\defeq}{\vcentcolon=}

\newcommand{\ul}[1]{\underline{#1}}
\newcommand{\ol}[1]{\overline{#1}}

\newcommand{\bdry}{\partial}

\newcommand{\subT}{{\text{$T$}}}
\newcommand{\subF}{{\text{$F$}}}
\newcommand{\subFT}{{\text{$\Fh[T]$}}}
\newcommand{\subh}{{\text{$h$}}}

\newcommand{\Mh}[1][h]{\calM_{#1}}
\newcommand{\Th}[1][h]{\calT_{#1}}
\newcommand{\Fh}[1][h]{\calF_{#1}}
\newcommand{\Fhb}{\Fh^{\rmb}}
\newcommand{\Fhi}{\Fh^{\rmi}}

\newcommand{\T}{{T}}
\newcommand{\F}{{F}}
\newcommand{\FT}{{\calF_\T}}

\newcommand{\hT}{h_\T}
 
\newcommand{\hX}{h_X}

\newcommand{\bdryT}{{\bdry \T}}

\newcommand{\matK}{\mat{K}}
\newcommand{\matKT}{\matK_\T}

\newcommand{\olK}{\ol{K}}
\newcommand{\ulK}{\ul{K}}
\newcommand{\olKT}{\olK_\T}
\newcommand{\ulKT}{\ulK_\T}

\newcommand{\alphaT}{\alpha_\T}

\newcommand{\nor}{\bmn}
\newcommand{\norT}{\nor_{\T}}
\newcommand{\norF}{\nor_{\F}}
\newcommand{\norTF}{\nor_{\T\F}}

\def\R{\bbR}

\newcommand{\POLY}[1]{\bbP^{#1}}

\newcommand{\HS}[1]{H^{#1}}
\newcommand{\HONE}{\HS{1}}
\newcommand{\HONEzr}{\HONE_0}

\newcommand{\LP}[1]{L^{#1}}
\newcommand{\LTWO}{\LP{2}}

\DeclareMathOperator*{\DIV}{div}

\newcommand{\norm}[2][]{\|#2\|_{#1}}
\newcommand{\seminorm}[2][]{|#2|_{#1}}
\newcommand{\brac}[2][]{(#2)_{#1}}

\newcommand{\SqBrac}[2][]{\Big[#2\Big]_{#1}}

\newcommand{\Brac}[2][]{\Big(#2\Big)_{#1}}

\newcommand{\SEMINORM}[2][]{\left|#2\right|_{#1}}
\newcommand{\BRAC}[2][]{\left(#2\right)_{#1}}

\newcommand{\piT}[1]{\pi_\subT^{#1}}
\newcommand{\piTzr}[1]{\piT{0, #1}}

\newcommand{\piFT}[1]{\pi_\subFT^{#1}}
\newcommand{\piFTzr}[1]{\piFT{0, #1}}

\newcommand{\piF}[1]{\pi_\subF^{#1}}
\newcommand{\piFzr}[1]{\piF{0, #1}}

\newcommand{\piKT}[1]{\pi_{\matK, T}^{#1}}
\newcommand{\piKTe}[1]{\piKT{1, #1}}

\newcommand{\piKh}[1]{\pi_{\matK, h}^{#1}}
\newcommand{\piKhe}[1]{\piKh{1, #1}}

\newcommand{\pKT}[1]{\rmp_{\matK,\subT}^{#1}}

\newcommand{\pKh}[1]{\rmp_{\matK,\subh}^{#1}}

\newcommand{\U}[2]{\ul{U}_{#1}^{#2}}

\newcommand{\UT}[1]{\U{\subT}{#1}}
\newcommand{\UTk}{\UT{k}}

\newcommand{\Uh}[1]{\U{h}{#1}}
\newcommand{\Uhk}{\Uh{k}}

\newcommand{\Uhkzr}{\U{h, 0}{k}}

\newcommand{\I}[2]{\ul{I}_{#1}^{#2}}

\newcommand{\IT}[1]{\I{\subT}{#1}}
\newcommand{\ITk}{\IT{k}}

\newcommand{\Ih}[1]{\I{h}{#1}}
\newcommand{\Ihk}{\Ih{k}}

\def\a{\rma}

\newcommand{\aKT}{\a_{\matK,T}}
\newcommand{\sKT}{\rms_{\matK,T}}
\newcommand{\aKh}{\a_{\matK,h}}
\newcommand{\sKh}{\rms_{\matK,h}}

\newcommand{\vT}{v_T}
\newcommand{\vh}{v_h}
\newcommand{\vF}{v_F}
\newcommand{\vFT}{v_{\FT}}

\newcommand{\wT}{w_T}

\newcommand{\wFT}{w_{\FT}}

\newcommand{\ulvT}{\ul{v}_T}
\newcommand{\uluT}{\ul{u}_T}
\newcommand{\ulwT}{\ul{w}_T}
\newcommand{\ulvh}{\ul{v}_h}
\newcommand{\uluh}{\ul{u}_h}

\newcommand{\locerrorRHS}[1]{\Big[\alphaT\hT^{k+1}\seminorm[\matK,\HS{k+2}(T)]{#1}\Big]^2}
\newcommand{\errorRHS}[1]{\BRAC{\sum_{T\in\Th}\locerrorRHS{#1}}^\frac12}

\newcommand{\logLogSlopeTriangle}[5]
{
	\pgfplotsextra
	{
		\pgfkeysgetvalue{/pgfplots/xmin}{\xmin}
		\pgfkeysgetvalue{/pgfplots/xmax}{\xmax}
		\pgfkeysgetvalue{/pgfplots/ymin}{\ymin}
		\pgfkeysgetvalue{/pgfplots/ymax}{\ymax}
		
		\pgfmathsetmacro{\xArel}{#1}
		\pgfmathsetmacro{\yArel}{#3}
		\pgfmathsetmacro{\xBrel}{#1-#2}
		\pgfmathsetmacro{\yBrel}{\yArel}
		\pgfmathsetmacro{\xCrel}{\xArel}
		
		\pgfmathsetmacro{\lnxB}{\xmin*(1-(#1-#2))+\xmax*(#1-#2)} 
		\pgfmathsetmacro{\lnxA}{\xmin*(1-#1)+\xmax*#1} 
		\pgfmathsetmacro{\lnyA}{\ymin*(1-#3)+\ymax*#3} 
		\pgfmathsetmacro{\lnyC}{\lnyA+#4*(\lnxA-\lnxB)}
		\pgfmathsetmacro{\yCrel}{\lnyC-\ymin)/(\ymax-\ymin)}
		
		\coordinate (A) at (rel axis cs:\xArel,\yArel);
		\coordinate (B) at (rel axis cs:\xBrel,\yBrel);
		\coordinate (C) at (rel axis cs:\xCrel,\yCrel);
		
		\draw[#5]   (A)-- node[pos=0.5,anchor=north] {\scriptsize{1}}
		(B)-- 
		(C)-- node[pos=0.,anchor=west] {\scriptsize{#4}} 
		cycle;
	}
}